\documentclass[12pt]{amsart}
\usepackage{amsmath, color}
\usepackage{amssymb}
\usepackage{amsfonts}
\usepackage{amsthm}
\usepackage{mathrsfs}
\usepackage[all]{xy}

\newcommand{\be}{\begin{equation}}
\newcommand{\ee}{\end{equation}}
\newtheorem{theorem}{Theorem}[section]
\newtheorem{lemma}[theorem]{Lemma}

\newtheorem{proposition}[theorem]{Proposition}

\theoremstyle{definition}
\newtheorem{definition}[theorem]{Definition}

\theoremstyle{remark}
\newtheorem{remark}[theorem]{Remark}

\numberwithin{equation}{section}

\begin{document}

\title{R-matrix realization of two-parameter quantum affine algebra $U_{r,s}(\widehat{\mathfrak{gl}}_n)$}
\author{Naihuan Jing}
\address{NJ: 
Department of Mathematics, North Carolina State University, Raleigh, NC 27695, USA}
\author{Ming Liu$^*$}
\address{ML: School of Mathematics, South China University of Technology,
Guangzhou 510640, China}

\thanks{{\scriptsize
\hskip -0.4 true cm MSC (2010): Primary: 17B30; Secondary: 17B68.
$*$Corresponding author.
}}

\maketitle

\begin{abstract}
We introduce the two-parameter quantum affine algebra $U_{r,s}(\widehat{\mathfrak{gl}}_n)$ via the RTT realization. The Drinfeld realization
is given and the type $A$ quantum affine algebra is proved to be a special subalgebra of our extended algebra.
\end{abstract}

\section{Introduction}

Quantum groups can be studied by two methods algebraically. The first approach was adopted by
Drinfeld \cite{D1,D3} and Jimbo \cite{J} to define the quantum enveloping algebra $U_q(\mathfrak{g})$ as a $q$-deformation of the
enveloping algebra $U(\mathfrak g)$
in terms of the Chevalley generators and Serre relations based on the data coming from the corresponding Cartan matrix.
For the Yangian algebra $Y(\mathfrak g)$ and the quantum affine algebra $U_q(\widehat{\mathfrak g})$, Drinfeld \cite{D2}
gave another realization called the new realization, which is analogue to the loop realization of the classical affine Lie algebra.
Using the Drinfeld realization, one can classify finite dimensional representations of quantum affine algebras and Yangians.

The second approach to quantum groups has its origin from the quantum inverse scattering method developed by the Leningrad school.
In \cite{FRT} Faddeev, Reshetikhin and Takhtajan have shown that both the quantum enveloping algebras $U_q(\mathfrak g)$ and the
dual quantum groups for finite classical simple Lie algebras $\mathfrak g$ can be studied in the RTT method using the solutions $R$ of the
Yang-Baxter equation:
\begin{equation}\label{YB Eq}
R_{12}R_{13}R_{23}=R_{23}R_{13}R_{12}.
\end{equation}
In \cite{FRT2}, the R-matrix realization of quantum loop algebras was also studied.
Later, Resheti\-khin and Semenov-Tian-Shansky \cite{RS} gave the central extension of the previous construction
in \cite{FRT2}, which can be viewed as the affine analogue of the construction in \cite{FRT}. In \cite{DF}, Ding and Frenkel
proved the isomorphism between the R-matrix realization and Drinfeld realization of quantum affine
algebras by using the Gauss decomposition of the generating matrix composed of elements of quantum affine algebras, thus
Ding-Frenkel's method provides a natural way to get the Drinfeld realization from the R-matrix realization
of the quantum affine algebra.

Two-parameter general linear and special linear quantum groups were considered by Takeuchi \cite{T}
using generators and relations.
The two-parameter quantum enveloping algebras have later gained attention after
 Benkart and Witherspoon's work \cite{BW2} on the $(r,s)$-deformed quantum algebras associated with
 $\mathfrak{gl}_n$ and $\mathfrak{sl}_n$, where the quantum R-matrix and
the Drinfeld doubles were obtained (see also \cite{BGH}).
Hu, Rosso and Zhang \cite{HRZ} first studied the Drinfeld realization of two-parameter quantum affine
algebra for type A and constructed its quantum Lyndon basis.
In \cite{JZ}, Zhang and one of us have realized the basic representations
of the two-parameter simply laced
quantum toroidal algebras in terms of the Grothendieck ring of certain deformed wreath products in
the context of the McKay correspondence.
Recently, Fan and Li \cite{FL} have given a geometric realization of
the negative part of the two-parameter quantum group $U_{r, s}(\mathfrak{g})$ in a uniformed manner.
In Hill and Wang's categorification
of the covering quantum group \cite{HW} there exists the second parameter $\pi$ subject to $\pi^2=1$, which in spirit would correspond
to some specialization of certain two-parameter quantum group.

A natural question can be asked whether the 2-parameter quantum affine algebras can be formulated in the general 
framework of the quantum scattering method.
And an even more important question is where on earth one should deform the
affine relations in a natural and canonical way. Corresponding to the geometric constructions in
the finite dimensional cases \cite{FL, HW}, it seems that the most natural answer should rely on
if one can reconstruct the quantum affine algebras using the RTT method such as that in \cite{BK,M} for the case of Yangians.
Every indicator points to that these two-parameter quantum affine algebras should be associated with certain
spectral parameter dependent R-matrices. For example, the authors have shown recently that the 2-parameter quantum algebra
$U_{r,s}(\mathfrak{gl}_n)$ and its dual are indeed realized by the RTT method \cite{JL2}.
If the affine case can also be confirmed realizable by the RTT method, it will be natural to define quantum affine root vectors
and the Hopf algebra structure. It will also help to re-establish the Drinfeld realization of
the 2-parameter quantum affine algebra $U_{r, s}(\widehat{\mathfrak sl}_n)$ as a subalgebra, which have potential applications in geometric realizations and other
applications in mathematical physics models.

In this paper, we answer these questions and show that the 2-parameter quantum affine algebras
are indeed realized by the Reshetikhin-Seminov-Tian-Shanski method.
 We  use certain spectral parametric $R$-matrix obtained by Yang-Baxterization \cite{GWX} to
introduce and study the two-parameter quantum affine algebra $U_{r,s}(\widehat{\mathfrak{gl}}_n)$ and show that it
contains the 2-parameter quantum
affine algebra $U_{r,s}(\widehat{\mathfrak{sl}}_n)$ as a subalgebra. Moreover,
using the Gauss decomposition of the generating matrix of the R-matrix realization of two-parameter quantum
affine algebra, we study the commutation relations of the Gaussian generators and get a natural Drinfeld
realization of two-parameter quantum affine algebra for both $U_{r,s}(\widehat{\mathfrak{gl}}_n)$ and $U_{r,s}(\widehat{\mathfrak{sl}}_n)$,
which provide a natural explanation for the quantum algebra, and in particular, the Drinfeld-Serre relations.

The paper is organized as follows. In section 2
we recall the basic results and the R-matrix of two-parameter quantum group $U_{r,s}(\widehat{\mathfrak{gl}}_n)$ .
In section 3, we give the definition of algebra $U(R)$ using the Reshetikhin-Semenov's method and study its Gauss decomposition in terms of quasi-determinants.
In section 4, we study the commutation relations between Gaussian generators and give the Drinfeld realization of $U_{r,s}(\widehat{\mathfrak{gl}}_n)$.
In section 5, we give the Drinfeld realization of $U_{r,s}(\widehat{\mathfrak{sl}}_n)$.

\section{Two-parameter quantum group $U_{r,s}(\mathfrak{gl}_n)$}


We first recall the Drinfeld-Jimbo form of the two parameter
quantum algebras.  Let $\Pi=\{\alpha_j=\epsilon_j-\epsilon_{j+1}| j=1,2,...n-1\}$ and $\Phi=\{\epsilon_i-\epsilon_j|1\leq i\neq j\leq n\}$
be the root system and the set of simple roots of type $A_{n-1}$,
where $\{\epsilon_i\}$ is an orthonormal basis of $\mathbb C^n$.

\begin{definition}
$U_{r,s}(\mathfrak{gl}_n)$ is a unital associated algebra over $\mathbb{C}$ generated by
$e_j$, $f_j$, $(1\leq j< n)$, and $a_i^{\pm 1}$, $b_i^{\pm 1}$ $(1\leq i\leq n)$, and satisfy the following relations.
\begin{description}
  \item[R1] Commuting elements $a_i^{\pm 1}$, $b_i^{\pm 1}$ $(1\leq i\leq n)$ and
  $a_ia_i^{-1}=b_ib_i^{-1}=1$,

  \item[R2] $a_ie_j=r^{\langle \epsilon_i,\alpha_j\rangle}e_ja_i$, and $a_if_j=r^{-\langle \epsilon_i,\alpha_j\rangle}f_ja_i$,
  \item[R3] $b_ie_j=s^{\langle \epsilon_i,\alpha_j\rangle}e_jb_i$, and $b_if_j=s^{-\langle \epsilon_i,\alpha_j\rangle}f_jb_i$,

  \item[R4] $[e_i,f_j]=\frac{\delta_{ij}}{r-s}(a_ib_{i+1}-a_{i+1}b_i)$,
  \item[R5] $[e_i,e_j]=[f_i,f_j]=0$ if $|i-j|>1$,
  \item[R6] $e_i^2e_{i+1}-(r+s)e_ie_{i+1}e_i+rse_{i+1}e_i^2=0$,\\
            $e_{i+1}^2e_{i}-(r+s)e_{i+1}e_{i}e_{i+1}+rse_{i}e_{i+1}^2=0$,
  \item[R7] $f_i^2f_{i+1}-(r^{-1}+s^{-1})f_if_{i+1}f_i+r^{-1}s^{-1}f_{i+1}f_i^2=0$,\\
            $f_{i+1}^2f_{i}-(r^{-1}+s^{-1})f_{i+1}f_{i}f_{i+1}+r^{-1}s^{-1}f_{i}f_{i+1}^2=0$.

\end{description}
\end{definition}

The algebra $U_{r,s}(\mathfrak{sl}_n)$ is the subalgebra of $U_{r,s}(\mathfrak{gl}_n)$
generated by $e_j$, $f_j$ and $\omega_j$, $\omega'_j$ $(1\leq j<n)$, where
$\omega_j=a_jb_{j+1}$, $\omega'_j=a_{j+1}b_j$. These elements satisfy the relations (R5-R7) along with
\begin{description}
  \item[R'1] The $\omega_i^{\pm 1}$, $\omega'^{\pm 1}_j$ $(1\leq i,j< n)$ all commutate one another and
  $\omega_i\omega_i^{-1}=\omega'_i\omega'^{-1}_i=1$,

  \item[R'2] $\omega_ie_j=r^{\langle \epsilon_i,\alpha_j\rangle}s^{\langle \epsilon_{i+1},\alpha_j\rangle}e_j\omega_i$,
  and $\omega_if_j=r^{-\langle \epsilon_i,\alpha_j\rangle}s^{-\langle \epsilon_{i+1},\alpha_j\rangle}f_j\omega_i$,
  \item[R'3]  $\omega'_ie_j=r^{\langle \epsilon_{i+1},\alpha_j\rangle}s^{\langle \epsilon_{i},\alpha_j\rangle}e_j\omega'_i$,
  and $\omega'_if_j=r^{-\langle \epsilon_{i+1},\alpha_j\rangle}s^{-\langle \epsilon_{i},\alpha_j\rangle}f_j\omega'_i$,

  \item[R'4] $[e_i,f_j]=\frac{\delta_{ij}}{r-s}(\omega_i-\omega'_i)$.

\end{description}

%

The natural representation $V=\mathbb{C}^n$ of $U_{r,s}(\mathfrak{gl}_n)$ is given by the following action:
\begin{align*}
e_i&=E_{i,i+1}, f_{i}=E_{i+1,i}, \\
a_j&=rE_{jj}+\sum_{k\neq j}E_{kk},\\
b_j&=sE_{jj}+\sum_{k\neq j}E_{kk},
\end{align*}
where $1\leq i<n, 1\leq j\leq n$, and $E_{ij}$ are the unit matrices of size $n\times n$. Corresponding to the natural
representation $V$, Benkart and Witherspoon gave the following matrix $\hat{R}=\hat{R}_{VV}$ \cite{BW}:
\begin{equation}
\hat{R}=\sum_{i=1}^{n}E_{ii}\otimes E_{ii}+r\sum_{i<j}E_{ji}\otimes E_{ij}+s^{-1}\sum_{i<j}E_{ij}\otimes E_{ji}+(1-rs^{-1})\sum_{i<j}E_{jj}\otimes E_{ii},
\end{equation}
satisfying the braiding relation on the tensor power $V^{\otimes k}$:
\begin{align*}
\hat{R}_i\circ \hat{R}_{i+1}\circ\hat{R}_i&=\hat{R}_{i+1}\circ \hat{R}_{i}\circ \hat{R}_{i+1}, \qquad\mbox{for $1\leq i<k$},\\
\hat{R}_i\circ \hat{R}_{j}&=\hat{R}_{j}\circ \hat{R}_{i}, \qquad\mbox{for $|i-j|\geq 2$},
\end{align*}
where $\hat{R}_i=(id_V)^{i-1}\otimes \hat{R}\otimes (id_V)^{k-1}$.
It is easy to see that $\hat{R}$ satisfies the Hecke relation
\begin{equation}\label{minimal polynomial}
(\hat{R}-1)(\hat{R}+rs^{-1})=0.
\end{equation}

In \cite{GWX}, the authors presented the so-called Yang-Baxterization to construct the corresponding braiding relation with a spectral parameter.
Suppose the braiding $\hat{R}$ has two eigenvalues $\lambda_1, \lambda_2$, then the Yang-Baxterization of $\hat{R}$ recovers the
associated spectral parameter dependent braid group representation $\hat{R}(z)$
via
\begin{equation}\label{Yang-Baxterization}
\hat{R}(z)=\lambda_2^{-1}\hat{R}+z\lambda_1\hat{R}^{-1}.
\end{equation}
which satisfies the relation
$$
\hat{R}_1(z)\hat{R}_2(zw)\hat{R}_1(w)=\hat{R}_2(w)\hat{R}_1(zw)\hat{R}_2(z).
$$

Using the fact that $\hat{R}=\hat{R}_{VV}$ has eigenvalues 1 and $-rs^{-1}$ on $V\otimes V$, we obtain that
\begin{proposition}\cite{JZL}
For the braid group representation $\hat{R}=\hat{R}_{VV}$, the $\hat{R}(z)$ is given by
\begin{align}\nonumber
\hat{R}(z)&=(1-zrs^{-1})\sum_{i=1}^nE_{ii}\otimes E_{ii}+(1-z)(r\sum_{i>j}+s^{-1}\sum_{i<j})E_{ij}\otimes E_{ji}\\ \label{R(z)}
&+z(1-rs^{-1})\sum_{i<j}E_{ii}\otimes E_{jj}+(1-rs^{-1})\sum_{i>j}E_{ii}\otimes E_{jj}.
\end{align}
\end{proposition}

\begin{remark}
Consider the R-matrix $R(z)=\frac{1}{1-zrs^{-1}}P\hat{R}(z)$, where $P=\sum_{ij}E_{ij}\otimes E_{ji}$.
\begin{eqnarray*}
R(z)&=&\sum_{i=1}^nE_{ii}\otimes E_{ii}+\frac{(1-z)r}{1-zrs^{-1}}\sum_{i>j}E_{ii}\otimes E_{jj}+\frac{(1-z)s^{-1}}{1-zrs^{-1}}\sum_{i<j}E_{ii}\otimes E_{jj}\\
&+&\frac{1-rs^{-1}}{1-zrs^{-1}}\sum_{i>j}E_{ij}\otimes E_{ji}+\frac{(1-rs^{-1})z}{1-rs^{-1}z}\sum_{i<j}E_{ij}\otimes E_{ji}.
\end{eqnarray*}
It is easy to check that $R(z)$ satisfy the quantum Yang-Baxter equation:
$$R_{12}(z)R_{13}(zw)R_{23}(w)=R_{23}(w)R_{13}(zw)R_{12}(z),$$
and the unitary condition:
\begin{equation}\label{unitary condition}
R_{21}(z)R(z^{-1})=R(z^{-1})R_{21}(z)=1.
\end{equation}
\end{remark}

It is clear that when $r=q, s=q^{-1}$ the $R$-matrix degenerates to the usual spectral dependent R-matrix \cite{DF}.

\section{The R-matrix algebra $U(R)$ and its Gauss decomposition}
In this section, we construct an RTT realization
of $U_{r,s}(\widehat{\mathfrak{gl}}_n)$ using Reshetikhin-Semenov's method. Furthermore, we give the Gauss decomposition of $U_{r,s}(\widehat{\mathfrak{gl}}_n)$ in terms of quasi-determinants \cite{GR}.
\subsection{The R-matrix algebra $U(R)$}
\begin{definition}
$U(R)$ is an associative algebra with generators  $l^{\pm}_{ij}[\mp m]$ , $m\in \mathbb{Z}_{+}\setminus0$
and $l^{+}_{kl}[0]$, $l^{-}_{lk}[0]$, $1\leq l\leq k\leq n$ and central elements $r^c, s^c$. Let $l^{\pm}_{ij}(z)=\sum_{m=0}^{\infty}l^{\pm}_{ij}[\mp m]z^{\pm m}$,
where $l_{kl}^{+}[0]=l^{-}_{lk}[0]=0$, for $1\leq k <l\leq n$. Let $L^{\pm}(z)=\sum_{i,j=1}^{n}E_{ij}\otimes l_{ij}^{\pm}(z)$.
Then the relations are given by the  following matrix equations on $End(V^{\otimes 2})\otimes U(R)$:
\begin{equation}
l^{+}_{ii}[0]l_{ii}^{-}[0]=l_{ii}^{-}[0]l^{+}_{ii}[0]
\end{equation}
\begin{equation}\label{generating relation1}
R(\frac{z}{w})L^{\pm}_{1}(z)L^{\pm}_2(w)=L^{\pm}_2(w)L^{\pm}_{1}(z)R(\frac{z}{w})
\end{equation}
\begin{equation}\label{generating relation2}
R(\frac{z_{+}}{w_{-}})L^{+}_1(z)L^{-}_2(w)=L^{-}_2(w)L^{+}_1(z)R(\frac{z_{-}}{w_{+}})
\end{equation}
where $z_+=zr^{\frac{c}{2}}$ and $z_{-}=zs^{\frac{c}{2}}$. Here Eq. (\ref{generating relation1}) are expanded in the
direction of either $\frac zw$ or $\frac wz$, and Eq. (\ref{generating relation2}) is expanded in the
direction of $\frac zw$.
\end{definition}

\begin{remark}
From the Eq (\ref{generating relation2}) and the unitary condition of R-matrix (\ref{unitary condition}), we have
\begin{equation}\label{Eq4.5}
R(\frac{z_{\pm}}{w_{\mp}})L^{\pm}_1(z)L^{\mp}_2(w)=L^{\mp}_2(w)L^{\pm}_1(z)R(\frac{z_{\mp}}{w_{\pm}})
\end{equation}
So the generating relations (\ref{generating relation1}), (\ref{generating relation2}) are equivalent to
the following:
\begin{equation}\label{equivalent generating relation1}
L^{\pm}_{1}(z)^{-1}L^{\pm}_2(w)^{-1}R(\frac{z}{w})=R(\frac{z}{w})L^{\pm}_2(w)^{-1}L^{\pm}_{1}(z)^{-1},
\end{equation}
\begin{equation}\label{equivalent generating relation2}
L^{\pm}_{1}(z)^{-1}L^{\mp}_2(w)^{-1}R(\frac{z_{\pm}}{w_{\mp}})=R(\frac{z_{\mp}}{w_{\pm}})L^{\mp}_2(w)^{-1}L^{\pm}_{1}(z)^{-1}.
\end{equation}
They are also equivalent to
\begin{equation}\label{equivalent generating relation1'}
L^{\pm}_2(w)^{-1}R(\frac{z}{w})L^{\pm}_{1}(z)=L^{\pm}_{1}(z)R(\frac{z}{w})L^{\pm}_2(w)^{-1},
\end{equation}
\begin{equation}\label{equivalent generating relation2'}
L^{\mp}_2(w)^{-1}R(\frac{z_{\pm}}{w_{\mp}})L^{\pm}_{1}(z)=L^{\pm}_{1}(z)R(\frac{z_{\mp}}{w_{\pm}})L^{\mp}_2(w)^{-1}.
\end{equation}

\end{remark}
\begin{remark} Relations (\ref{generating relation1}) and (\ref{Eq4.5}) can be equivalently written in terms of the generating series:
\begin{align}\notag
&\left(\delta_{pr}+(rs\delta_{p>r}+\delta_{p<r})\frac{w-z}{ws-zr}\right)l^{\pm}_{pq}(z)l^{\pm}_{rs}(w)
+\frac{s-r}{ws-zr}(z\delta_{p<r}
+w\delta_{p>r})l_{rq}^{\pm}(z)l_{ps}^{\pm}(w)=\\ \label{Eq4.9}
&\ \left(\delta_{qs}+(rs{\delta_{q>s}}+\delta_{q<s})\frac{w-z}{ws-zr}\right)l^{\pm}_{rs}(w)l^{\pm}_{pq}(z)
+\frac{s-r}{ws-zr}(w\delta_{q<s}
+z\delta_{q>s})l_{rq}^{\pm}(w)l_{ps}^{\pm}(z),
\end{align}
\begin{align}\notag
\frac{s-r}{w_{\mp}s-z_{\pm}r}&\big((z_{\pm}\delta_{p<r}
+w_{\mp}\delta_{p>r})-(w_{\pm}\delta_{q<s}
+z_{\mp}\delta_{q>s})\big)l_{rq}^{\pm}(z)l_{ps}^{\mp}(w)\\ \label{Eq4.10}
=-&\left(\delta_{pr}+(rs{\delta_{p>r}}+\delta_{p<r})\frac{w_{\mp}-z_{\pm}}{w_{\mp}s-z_{\pm}r}\right)l^{\pm}_{pq}(z)l^{\mp}_{rs}(w)\\ \notag
+&\left(\delta_{qs}+(rs{\delta_{q>s}}+\delta_{q<s})\frac{w_{\pm}-z_{\mp}}{w_{\pm}s-z_{\mp}r}\right)l^{\mp}_{rs}(w)l^{\pm}_{pq}(z).
\end{align}

Similarly Eqs. (\ref{equivalent generating relation1}--\ref{equivalent generating relation2}) are equivalent to the following
\begin{align}\notag
&\left(\delta_{qs}+(rs{\delta_{q>s}}+\delta_{q<s})\frac{w-z}{ws-zr}\right)l^{\pm}_{pq}(z)'l^{\pm}_{rs}(w)'
+\frac{s-r}{ws-zr}(z\delta_{q>s}
+w\delta_{q<s})l_{ps}^{\pm}(z)'l_{rq}^{\pm}(w)'=\\ \label{Eq4.11}
&\quad \left(\delta_{pr}+(rs{\delta_{p>r}}+\delta_{p<r})\frac{w-z}{ws-zr}\right)l^{\pm}_{rs}(w)'l^{\pm}_{pq}(z)'
+\frac{s-r}{ws-zr}(w\delta_{p>r}
+z\delta_{p<r})l_{ps}^{\pm}(w)'l_{rq}^{\pm}(z)',
\end{align}
\begin{align}\notag
&\left(\delta_{qs}+((rs){\delta_{q>s}}+\delta_{q<s})\frac{w_{\mp}-z_{\pm}}{w_{\mp}s-z_{\pm}r}\right)l^{\pm}_{pq}(z)'l^{\mp}_{rs}(w)'
+\frac{s-r}{w_{\mp}s-z_{\pm}r}(z_{\pm}\delta_{q>s}
+w_{\mp}\delta_{q<s})l_{ps}^{\pm}(z)'l_{rq}^{\mp}(w)'=\\ \label{Eq4.12}
&\left(\delta_{pr}+((rs){\delta_{p>r}}+\delta_{p<r})\frac{w_{\pm}-z_{\mp}}{w_{\pm}s-z_{\mp}r}\right)l^{\mp}_{rs}(w)'l^{\pm}_{pq}(z)'
+\frac{s-r}{w_{\pm}s-z_{\mp}r}(w_{\pm}\delta_{p>r}
+z_{\mp}\delta_{p<r})l_{ps}^{\mp}(w)'l_{rq}^{\pm}(z)'.
\end{align}

Finally Eqs. (\ref{equivalent generating relation1'}--\ref{equivalent generating relation2'}) are seen to be equivalent to the following
\begin{align}\label{Eq4.13}\notag
&\delta_{ps}l_{rs}^{\pm}(w)'l_{pq}^{\pm}(z)+((rs){\delta_{p>s}}+\delta_{p<s})\frac{w-z}{ws-zr}l^{\pm}_{rs}(w)'l^{\pm}_{pq}(z)\\ \notag
&+\frac{s-r}{ws-zr}
\delta_{ps}(w\sum_{j<p}l^{\pm}_{rj}(w)'l^{\pm}_{jq}(z)+z\sum_{j>p}l^{\pm}_{rj}(w)'l^{\pm}_{jq}(z))=\\
&\delta_{qr}l_{pq}^{\pm}(z)l_{rs}^{\pm}(w)'+((rs){\delta_{q>r}}+\delta_{q<r})\frac{w-z}{ws-zr}l^{\pm}_{pq}(z)l^{\pm}_{rs}(w)'\\ \notag
&+\frac{s-r}{ws-zr}
\delta_{qr}(w\sum_{i>q}l^{\pm}_{pi}(z)l^{\pm}_{is}(w)'+z\sum_{i<q}l^{\pm}_{pi}(z)l^{\pm}_{is}(w)'), \\ \notag
&\delta_{ps}l_{rs}^{\mp}(w)'l_{pq}^{\pm}(z)+((rs){\delta_{p>s}}+\delta_{p<s})\frac{w_{\mp}-z_{\pm}}{w_{\mp}s-z_{\pm}r}l^{\mp}_{rs}(w)'l^{\pm}_{pq}(z)
\\ \notag
&+\frac{s-r}{w_{\mp}s-z_{\pm}r}
\delta_{ps}(w_{\mp}\sum_{j<p}l^{\mp}_{rj}(w)'l^{\pm}_{jq}(z)+z_{\pm}\sum_{j>p}l^{\mp}_{rj}(w)'l^{\pm}_{jq}(z))=\\ \label{Eq4.14}
&\delta_{qr}l_{pq}^{\pm}(z)l_{rs}^{\mp}(w)'+((rs){\delta_{q>r}}+\delta_{q<r})\frac{w_{\pm}-z_{\mp}}{w_{\pm}s-z_{\mp}r}l^{\pm}_{pq}(z)l^{\mp}_{rs}(w)'
\\ \notag
&+\frac{s-r}{w_{\pm}s-z_{\mp}r}
\delta_{qr}(w_{\pm}\sum_{i>q}l^{\pm}_{pi}(z)l^{\mp}_{is}(w)'+z_{\mp}\sum_{i<q}l^{\pm}_{pi}(z)l^{\mp}_{is}(w)').
\end{align}
\end{remark}
\subsection{Quasi-determinant and Gauss decomposition}
\begin{definition}
Let $X$ be a square matrix over a ring with identity such that its inverse matrix $X^{-1}$ exists. Suppose that the $(j,i)$-th entry of
$X^{-1}$ is an invertible element of the ring. The $(i,j)$-th quasi-determinant $\left |X\right |_{ij}$ of $X$ is defined by
\begin{equation*}
\left |X\right |_{ij}=\left |
                        \begin{array}{ccccc}
                          x_{11} & \cdots & x_{1j} & \cdots & x_{1n}\\
                           & \cdots &  & \cdots &  \\
                          x_{i1} & \cdots & \boxed{x_{ij}} & \cdots & x_{in} \\
                           & \cdots &  & \cdots &  \\
                          x_{n1} & \cdots & x_{nj} & \cdots & x_{nn} \\
                        \end{array}
                      \right |=\left ((X^{-1})_{ji}\right )^{-1}.
\end{equation*}
\end{definition}

By \cite[Th. 4.9.6]{GGRW}, we have the following decomposition of $L^{\pm}(z)$:
\begin{proposition}
$L^{\pm}(z)$ have the following unique decomposition:
\begin{equation}
L^{\pm}(z)=F^{\pm}(z)K^{\pm}(z)E^{\pm}(z),
\end{equation}
where
\begin{equation}
F^{\pm}(z)=\begin{pmatrix}
       1 &  &  & 0 \\
       f_{21}^{\pm}(z) & \ddots &  &  \\
       f_{31}^{\pm}(z) & \ddots & \ddots & \ddots \\
       \vdots & \ddots &  &  \\
       f_{n1}^{\pm}(z) & \ldots & f_{n,n-1}^{\pm}(z) & 1 \\
     \end{pmatrix},
\end{equation}
\begin{equation}
K^{\pm}(z)=\begin{pmatrix}
       k_1^{\pm}(z) &  &  & 0 \\
        & \ddots &  &  \\
        &  & \ddots &  \\
       0 &  &  & k_{n}^{\pm}(z) \\
     \end{pmatrix},
\end{equation}
\begin{equation}
E^{\pm}(z)=\begin{pmatrix}
       1 & e_{12}^{\pm}(z) & e_{13}^{\pm}(z) & \ldots & e_{1n}^{\pm}(z) \\
        & \ddots & \ddots & \ddots & \vdots \\
        &  &  &  & e_{n-1,n}^{\pm}(z) \\
       0 &  &  &  & 1 \\
     \end{pmatrix},
\end{equation}
and  for $1\leqslant i\leqslant n$ and $1\leqslant i< j\leqslant n$
$$k_i^{\pm}(z)=\left|
                 \begin{array}{cccc}
                   l^{\pm}_{11}(z) & \cdots & l^{\pm}_{1,i-1}(z) & l^{\pm}_{1i}(z) \\
                   \vdots & \ddots &  & \vdots \\
                   l^{\pm}_{i1}(z) & \ldots & l^{\pm}_{i,i-1}(z) & \boxed{l^{\pm}_{ii}(z)} \\
                 \end{array}
               \right| =
\sum_{m\in \mathbb{Z}_+}k_i^{\pm}(\mp m)z^{\pm m},$$
$$e_{ij}^{\pm}(z)=k_i^{\pm}(z)^{-1}\left|
                 \begin{array}{cccc}
                   l^{\pm}_{11}(z) & \cdots & l^{\pm}_{1,i-1}(z) & l^{\pm}_{1j}(z) \\
                   \vdots & \ddots & \vdots & \vdots \\
                   l^{\pm}_{i-1,1}(z) & \ldots & l^{\pm}_{i-1,i-1}(z) & l^{\pm}_{i-1,j}(z) \\
                   l^{\pm}_{i1}(z) & \ldots & l^{\pm}_{i,i-1}(z) & \boxed{ l^{\pm}_{ij}(z)} \\
                 \end{array}
               \right|=\sum_{m\in \mathbb{Z}_+}e_{ij}^{\pm}(\mp m)z^{\pm m},$$
$$f_{ji}^{\pm}(z)=\left|
                 \begin{array}{cccc}
                   l^{\pm}_{11}(z) & \cdots & l^{\pm}_{1,i-1}(z) & l^{\pm}_{1i}(z) \\
                   \vdots & \ddots & \vdots & \vdots \\
                   l^{\pm}_{i-1,1}(z) & \ldots & l^{\pm}_{i-1,i-1}(z) & l^{\pm}_{i-1,i}(z) \\
                   l^{\pm}_{j1}(z) & \ldots & l^{\pm}_{j,i-1}(z) & \boxed{ l^{\pm}_{ji}(z)} \\
                 \end{array}
               \right|k_i^{\pm}(z)^{-1}=\sum_{m\in \mathbb{Z}_+}f_{ji}^{\pm}(\mp m)z^{\pm m}.$$
\end{proposition}

\section{Drinfeld realization of $U_{r,s}(\widehat{\mathfrak{gl}}_n)$}

In this section, we give the Drinfeld realization of $U_{r,s}(\widehat{\mathfrak{gl}}_n)$
by studying the commutation relations between Gaussian generators. This is done through
the RTT presentation of $U_{r,s}(\widehat{\mathfrak{gl}}_n)$.

\begin{theorem}\label{Theorem5.2}
Let $X_{i}^{+}(z)=e_{i,i+1}^{+}(z_{+})-e_{i,i+1}^{-}(z_{-})$,
and $X_{i}^{-}(z)=f_{i+1,i}^{+}(z_{-})-f_{i+1,i}^{-}(z_{+})$.
Then the following relations are satisfied in $U(R)$:
\begin{align*}
k_{j}^{+}[0]k_j^{-}[0]&=k_j^{-}[0]k_{j}^{+}[0], \\
k_{i}^{\pm}(z)k_{j}^{\pm}(w)&=k_{j}^{\pm}(w)k_{i}^{\pm}(z),\\
k^{+}_i(z)k^{-}_i(w)&=k^{-}_i(w)k^{+}_i(z),
\end{align*}
\begin{align*}
\frac{z_{\mp}-w_{\pm}}{z_{\mp}r-w_{\pm}s}k_{i}^{\mp}(z)k_{j}^{\pm}(w)&=
\frac{z_{\pm}-w_{\mp}}{z_{\pm}r-w_{\mp}s}k_{j}^{\pm}(w)k_{i}^{\mp}(z)  \quad\mbox{if $j>i$},\\
 k_{i}^{\pm}(w)^{-1}X_{j}^{\mp}(z)k^{\pm}_i(w)&=X^{\mp}_{i}(z),   \quad\mbox{$i-j\leq -1$}, or\quad\mbox{$i-j\geq 2$},\\
k_{i}^{\pm}(w)^{-1}X_{j}^{\pm}(z)k^{\pm}_i(w)&=X^{\pm}_{i}(z), \quad\mbox{$i-j\leq -1$}, or\quad\mbox{$i-j\geq 2$},
\end{align*}
\begin{align*}
k_{i}^{\pm}(z)^{-1}X_{i}^{+}(w)k_{i}^{\pm}(z)&=\frac{z_{\mp}r-ws}{z_{\mp}-w}X_{i}^{+}(w),\\
k_{i+1}^{\pm}(z)^{-1}X_{i}^{+}(w)k_{i+1}^{\pm}(z)&=\frac{z_{\mp}s-wr}{z_{\mp}-w}X_{i}^{+}(w),\\
k_{i}^{\pm}(z)X_{i}^{-}(w)k_{i}^{\pm}(z)^{-1}&=\frac{z_{\pm}r-ws}{z_{\pm}-w}X_{i}^{-}(w),\\
k_{i+1}^{\pm}(z)X_{i}^{-}(w)k_{i+1}^{\pm}(z)^{-1}&=\frac{z_{\pm}s-wr}{z_{\pm}-w}X_{i}^{-}(w),
\end{align*}
\begin{align*}
(zs-wr)X_{i}^{+}(z)X_{i}^{+}(w)&=(zr-ws)X_{i}^{+}(w)X_{i}^{+}(z),\\
(zr-ws)X_{i}^{-}(z)X_{i}^{-}(w)&=(zs-wr)X_{i}^{-}(w)X_{i}^{-}(z),\\
(z-w)X_{i}^{-}(z)X^{-}_{i+1}(w)&=(zr-ws)X^{-}_{i+1}(w)X_{i}^{-}(z),\\
(zr-ws)X_{i}^{+}(z)X^{+}_{i+1}(w)&=(z-w)X^{+}_{i+1}(w)X_{i}^{+}(z),
\end{align*}

\begin{align*}
[X^{\pm}_i(z),X^{\pm}_j(w)]&=0,~~ if  (\alpha_i, \alpha_j)=0, \\
[X^{+}_i(z),X^{-}_j(w)]&=(r-s)\delta_{ij}\left\{\delta\left(\frac{z_-}{w_+}\right)k_{i+1}^{-}(w_{+})k_{i}^{-}(w_{+})^{-1}\right.\\
&\hskip 0.8in -\left.\delta\left(\frac{z_+}{w_-}\right)k_{i+1}^{+}(z_{+})k_{i}^{+}(z_{+})^{-1}\right\}
\end{align*}
and the following Serre relations hold in $U(R)$:
\begin{align*}
&\{X_{i}^{-}(z_1)X_{i}^{-}(z_2)X_{i+1}^{-}(w)-(r+s)X_{i}^{-}(z_1)X_{i+1}^{-}(w)X_{i}^{-}(z_2)\\
&\qquad\quad+rsX_{i+1}^{-}(w)X_{i}^{-}(z_1)X_{i}^{-}(z_2)\}+\{z_1\leftrightarrow z_2\}=0\\
&\{rsX_{i+1}^{-}(z_1)X_{i+1}^{-}(z_2)X_{i}^{-}(w)-(r+s)X_{i+1}^{-}(z_1)X_{i}^{-}(w)X_{i+1}^{-}(z_2)\\
&\qquad\quad+X_{i}^{-}(w)X_{i+1}^{-}(z_1)X_{i+1}^{-}(z_2)
\}+\{z_1\leftrightarrow z_2\}=0\\
&\{rsX_{i}^{+}(z_1)X_{i}^{+}(z_2)X_{i+1}^{+}(w)-(r+s)X_{i}^{+}(z_1)X_{i+1}^{+}(w)X_{i}^{+}(z_2)\\
&\qquad\quad+X_{i+1}^{+}(w)X_{i}^{+}(z_1)X_{i}^{+}(z_2)
\}+\{z_1\leftrightarrow z_2\}=0.\\
&\{X_{i+1}^{+}(z_1)X_{i+1}^{+}(z_2)X_{i}^{+}(w)-(r+s)X_{i+1}^{+}(z_1)X_{i}^{+}(w)X_{i+1}^{+}(z_2)\\
&\qquad\quad+rsX_{i}^{+}(w)X_{i+1}^{+}(z_1)X_{i+1}^{+}(z_2)
\}+\{z_1\leftrightarrow z_2\}=0.
\end{align*}
where the formal delta function
$\delta(z)=\sum_{n\in \mathbb{Z}}z^n.$
\end{theorem}

The idea of the proof is to first check the relations for $n=2$ and $n=3$, and then use induction to show the general situation.

\subsection{Case of $n=2$}

We first check the theorem for the case of $n=2$. Let us consider the generators
$k^{\pm}_i(z)$, $e_{12}(z)$, $f_{21}(z)$, $i=1, 2$, we have that
\begin{lemma}\label{lemma5.3}
\begin{equation}\label{Eq5.5}
k_{1}^{\pm}(z)k_{2}^{\pm}(w)=k_{2}^{\pm}(w)k_{1}^{\pm}(z),
\end{equation}
\begin{equation}\label{Eq5.6}
k^{+}_i(z)k^{-}_i(w)=k^{-}_i(w)k^{+}_i(z),~~i=1,2,
\end{equation}
\begin{equation}\label{Eq5.7}
\frac{w_{\mp}-z_{\pm}}{w_{\mp}s-z_{\pm}r}k_{1}^{\pm}(z)k_{2}^{\mp}(w)
=\frac{w_{\pm}-z_{\mp}}{w_{\pm}s-z_{\mp}r}k_{2}^{\mp}(w)k_{1}^{\pm}(z).
\end{equation}
\end{lemma}

\begin{proof}
Here we only prove Eq. (\ref{Eq5.7}) as the other relations are shown similarly.
By Eq. (\ref{Eq4.14}), we get the following relations:
\begin{equation*}
\frac{w_{\pm}-z_{\mp}}{w_{\pm}s-z_{\mp}r}l_{11}^{\pm}(z)l_{22}'^{\mp}(w)=\frac{w_{\mp}-z_{\pm}}{w_{\mp}s-z_{\pm}r}l_{22}'^{\mp}(w)l_{11}^{\pm}(z).
\end{equation*}
Note that $l_{22}'^{\pm}(w)=k_2^{\pm}(w)^{-1}$, the above is equivalent to
\begin{equation*}
\frac{w_{\mp}-z_{\pm}}{w_{\mp}s-z_{\pm}r}k_{1}^{\pm}(z)k_{2}^{\mp}(w)=\frac{w_{\pm}-z_{\mp}}{w_{\pm}s-z_{\mp}r}k_{2}^{\mp}(w)k_{1}^{\pm}(z).
\end{equation*}
\end{proof}

The following lemma gives the relations between $k_{1}^{\pm}(z)$ and $e_{12}^{\pm}(z)$ or $f_{21}^{\pm}(z)$.
\begin{lemma}\label{lemma5.4}
\begin{equation}\label{Eq5.8}
k_{1}^{\pm}(z)e_{12}^{\pm}(w)=\frac{w-z}{ws-zr}e_{12}^{\pm}(w)k_{1}^{\pm}(z)+\frac{w(s-r)}{ws-zr}k_{1}^{\pm}(z)e_{12}^{\pm}(z)
\end{equation}

\begin{equation}\label{Eq5.9}
k_{1}^{\pm}(z)e_{12}^{\mp}(w)=\frac{w_{\pm}-z_{\mp}}{w_{\pm}s-z_{\mp}r}e_{12}^{\mp}(w)k_{1}^{\pm}(z)+\frac{w_{\pm}(s-r)}{w_{\pm}s-z_{\mp}r}k_{1}^{\pm}(z)e_{12}^{\pm}(z),
\end{equation}

\begin{equation}\label{Eq5.10}
f_{21}^{\pm}(w)k_{1}^{\pm}(z)=\frac{w-z}{ws-zr}k_{1}^{\pm}(z)f_{21}^{\pm}(w)+\frac{z(s-r)}{ws-zr}f_{21}^{\pm}(z)k_{1}^{\pm}(z),
\end{equation}

\begin{equation}\label{Eq5.11}
f_{21}^{\mp}(w)k_{1}^{\pm}(z)=\frac{w_{\mp}-z_{\pm}}{w_{\mp}s-z_{\pm}r}k_{1}^{\pm}(z)f_{21}^{\mp}(w)+\frac{z_{\pm}(s-r)}{w_{\mp}s-z_{\pm}r}f_{21}^{\pm}(z)k_{1}^{\pm}(z).
\end{equation}

\end{lemma}

\begin{proof} We prove Eqs. (\ref{Eq5.8}-\ref{Eq5.9}), and the relations for $f_{21}^{\pm}(w)$ and $k_1(z)^{\pm}$ can be similarly shown.

From Eqs. (\ref{Eq4.9}-\ref{Eq4.10}) it follows that
\begin{equation*}
l_{11}^{\pm}(z)l_{12}^{\pm}(w)=\frac{w-z}{ws-zr}l_{12}^{\pm}(w)l_{11}^{\pm}(z)+\frac{w(s-r)}{ws-zr}l_{11}^{\pm}(w)l_{12}^{\pm}(z),
\end{equation*}
\begin{equation*}
l_{11}^{\pm}(z)l_{12}^{\mp}(w)=\frac{w_{\pm}-z_{\mp}}{w_{\pm}s
-z_{\mp}r}l_{12}^{\mp}(w)l_{11}^{\pm}(z)+\frac{w_{\pm}(s-r)}{w_{\pm}s-z_{\mp}r}l_{11}^{\mp}(w)l_{12}^{\pm}(z).
\end{equation*}
Then Eq (\ref{Eq5.8}) and (\ref{Eq5.9}) hold by using Eqs. (\ref{Eq5.5}-\ref{Eq5.6}) and the fact that $k_1(w)$ is invertible.

\end{proof}
From Lemma \ref{lemma5.4}, we obtain the following proposition.
\begin{proposition}
\begin{equation}\label{Eq1_prop5.5}
k_{1}^{\pm}(z)X_{1}^{+}(w)=\frac{w_{\pm}-z}{w_{\pm}s-zr}X_{1}^{+}(w)k_{1}^{\pm}(z),
\end{equation}

\begin{equation}\label{Eq2_prop5.5}
X_{1}^{-}(w)k_{1}^{\pm}(z)=\frac{w_{\mp}-z}{w_{\mp}s-zr}k_{1
}^{\pm}(z)X_{1}^{-}(w).
\end{equation}

\end{proposition}
\begin{proof}
We only prove Eq. (\ref{Eq1_prop5.5}) to give an example.

From Eq. (\ref{Eq5.8}), we have
\begin{equation*}
k_{1}^{+}(z)e_{12}^{+}(w_{+})=\frac{w_+-z}{w_+s-zr}e_{12}^{+}(w_+)k_{1}^{+}(z)+\frac{w_+(s-r)}{w_+s-zr}k_{1}^{+}(z)e_{12}^{+}(z).
\end{equation*}
On the other hand, Eq. (\ref{Eq5.9}) gives that
\begin{eqnarray*}
k_{1}^{+}(z)e_{12}^{-}(w_{-})
&=&\frac{w_{+}-z}{w_{+}s-zr}e_{12}^{-}(w_{-})k_{1}^{+}(z)+\frac{w_{+}(s-r)}{w_{+}s-zr}k_{1}^{+}(z)e_{12}^{+}(z).
\end{eqnarray*}
Taking the difference, we prove the identity:
\begin{equation*}
k_{1}^{+}(z)X_{1}^{+}(w)=\frac{w_{+}-z}{w_{+}s-zr}X_{1}^{+}(w)k_{1}^{+}(z).
\end{equation*}
Similarly, we can prove the other identity in Eq. (\ref{Eq1_prop5.5}). \end{proof}

In the following lemma, the relations between $k_2(z)$ and $e_{12}(z)$ or $f_{21}(z)$ are given.
\begin{lemma}\label{lemma5.6}
\begin{equation}\label{Eq5.12}
e_{12}^{\pm}(z)k_{2}^{\pm}(w)^{-1}=\frac{w-z}{ws-zr}k_{2}^{\pm}(w)^{-1}e_{12}^{\pm}(z)+\frac{z(s-r)}{ws-zr}e_{12}^{\pm}(w)k_{2}^{\pm}(w)^{-1}
\end{equation}

\begin{equation}\label{Eq5.13}
e_{12}^{\pm}(z)k_{2}^{\mp}(w)^{-1}=\frac{w_{\pm}-z_{\mp}}{w_{\pm}s-z_{\mp}r}k_{2}^{\mp}(w)^{-1}e_{12}^{\pm}(z)+\frac{z_{\mp}(s-r)}{w_{\pm}s-z_{\mp}r}e_{12}^{\pm}(w)k_{2}^{\pm}(w)^{-1}
\end{equation}
\begin{equation}
k_{2}^{\pm}(w)^{-1}f_{21}^{\pm}(z)=\frac{w-z}{ws-zr}f_{21}^{\pm}(z)k_{2}^{\pm}(w)^{-1}+\frac{w(s-r)}{ws-zr}f_{21}^{\pm}(w)k_{2}^{\pm}(w)^{-1},
\end{equation}

\begin{equation}\label{Eq4_lemma5.6}
k_{2}^{\mp}(w)^{-1}f_{21}^{\pm}(z)=\frac{w_{\mp}-z_{\pm}}{w_{\mp}s-z_{\pm}r}f_{21}^{\pm}(z)k_{2}^{\mp}(w)^{-1}+\frac{w_{\mp}(s-r)}{w_{\mp}s-z_{\pm}r}f_{21}^{\mp}(w)k_{2}^{\mp}(w)^{-1},
\end{equation}

\end{lemma}

\begin{proof} Here we just prove Eq. (\ref{Eq4_lemma5.6}) to show the idea.

From the relations in (\ref{Eq4.12}), we have

\begin{equation*}
l_{22}'^{\mp}(w)l_{21}'^{\pm}(z)=\frac{w_{\mp}-z_{\pm}}{w_{\mp}s-z_{\pm}r}l_{21}'^{\pm}(z)l_{22}'^{\mp}(w)+\frac{w_{\mp}(s-r)}{w_{\mp}s-z_{\pm}r}l_{22}'^{\pm}(z)l_{21}'^{\mp}(w),
\end{equation*}

Then the Gauss decomposition and  Eq. (\ref{Eq5.6}) imply  Eq. (\ref{Eq4_lemma5.6}).
\end{proof}

The following result is a consequence of Lemma \ref{lemma5.6}.
\begin{proposition}
\begin{equation}\label{Eq1_prop5.7}
k_{2}^{\pm}(w)^{-1}X_{1}^{+}(z)k_{2}^{\pm}(w)=\frac{ws-z_{\pm}r}{w-z_{\pm}}X_{1}^{+}(z),
\end{equation}

\begin{equation}\label{Eq2_prop5.7}
k_{2}^{\pm}(w)^{-1}X_{1}^{-}(z)k_{2}^{\pm}(w)=\frac{w-z_{\mp}}{ws-z_{\mp}r}X_{1}^{-}(z).
\end{equation}

\end{proposition}

\begin{proof} As the two relations are similar, we prove Eq. (\ref{Eq1_prop5.7}). In fact, using
Eqs. (\ref{Eq5.12}-\ref{Eq5.13}), we have
\begin{equation*}
k_2^{+}(w)^{-1}e_{12}^{+}(z_+)k_2^{+}(w)=\frac{ws-z_{+}r}{w-z_{+}}e_{12}^{+}(z_{+})-\frac{z_{+}(s-r)}{w-z_{+}}e_{12}^{+}(w)
\end{equation*}
and
\begin{equation*}
k_2^{+}(w)^{-1}e_{12}^{+}(z_+)k_2^{+}(w)=\frac{ws-z_{+}r}{w-z_{+}}e_{12}^{+}(z_{+})-\frac{z_{+}(s-r)}{w-z_{+}}e_{12}^{+}(w)
\end{equation*}

\end{proof}

The relations between $e_{12}^{\pm}(z)$ and $e_{12}^{\pm}(w)$ are given in the following lemma.
\begin{lemma}\label{lemma5.8}
\begin{equation}\label{Eq1_lamma5.8}
\frac{zs-wr}{z-w}e_{12}^{\pm}(z)e_{12}^{\pm}(w)-\frac{z(s-r)}{z-w}e_{12}^{\pm}(w)^{2}
=\frac{ws-zr}{w-z}e_{12}^{\pm}(w)e_{12}^{\pm}(z)-\frac{w(s-r)}{w-z}e_{12}^{\pm}(z)^{2}
\end{equation}

\begin{equation}\label{Eq2_lemma5.8}
\frac{z_{\mp}s-w_{\pm}r}{z_{\mp}-w_{\pm}}e_{12}^{\pm}(z)e_{12}^{\mp}(w)-\frac{z_{\mp}(s-r)}{z_{\mp}-w_{\pm}}e_{12}^{\mp}(w)^{2}
=\frac{w_{\pm}s-z_{\mp}r}{w_{\pm}-z_{\mp}}e_{12}^{\mp}(w)e_{12}^{\pm}(z)-\frac{w_{\pm}(s-r)}{w_{\pm}-z_{\mp}}e_{12}^{\pm}(z)^{2}
\end{equation}

\begin{equation}
\frac{ws-zr}{w-z}f_{21}^{\pm}(z)f_{21}^{\pm}(w)-\frac{z(s-r)}{w-z}f_{21}^{\pm}(z)^{2}
=\frac{zs-wr}{z-w}f_{21}^{\pm}(w)f_{21}^{\pm}(z)-\frac{w(s-r)}{z-w}f_{21}^{\pm}(w)^{2}
\end{equation}

\begin{equation}
\frac{w_{\mp}s-z_{\pm}r}{w_{\mp}-z_{\pm}}f_{21}^{\pm}(z)f_{21}^{\mp}(w)-\frac{z_{\pm}(s-r)}{w_{\mp}-z_{\pm}}f_{21}^{\pm}(z)^{2}
=\frac{z_{\pm}s-w_{\mp}r}{z_{\pm}-w_{\mp}}f_{21}^{\mp}(w)f_{21}^{\pm}(z)-\frac{w_{\mp}(s-r)}{z_{\pm}-w_{\mp}}f_{21}^{\pm}(w)^{2}
\end{equation}

\end{lemma}
\begin{proof}
We also only prove Eq. (\ref{Eq2_lemma5.8}) since the other relations are obtained similarly.
Using Eq. (\ref{Eq4.10}), we have that
\begin{equation*}
l^{\pm}_{12}(z)l^{\mp}_{12}(w)=l^{\mp}_{12}(w)l^{\pm}_{12}(z).
\end{equation*}
Using the Gauss decomposition, we have
\begin{equation}\label{Eq5.18}
k^{\pm}_{1}(z)e^{\pm}_{12}(z)k^{\mp}_{1}(w)e^{\mp}_{12}(w)=k^{\mp}_{1}(w)e^{\mp}_{12}(w)k^{\pm}_{1}(z)e^{\pm}_{12}(z).
\end{equation}
By Eq. (\ref{Eq4.10}), we have
\begin{equation*}
l^{\pm}_{12}(z)l^{\mp}_{11}(w)=\frac{(w_{\pm}-z_{\mp})rs}{w_{\pm}s-z_{\mp}r}l^{\mp}_{11}(w)l^{\pm}_{12}(z)
+\frac{z_{\mp}(s-r)}{w_{\pm}s-z_{\mp}r}l^{\mp}_{12}(w)l^{\pm}_{11}(z),
\end{equation*}
which is equivalent to
\begin{equation*}
k^{\pm}_1(z)e^{\pm}_{12}(z)k^{\mp}_1(w)=\frac{(w_{\pm}-z_{\mp})rs}{w_{\pm}s-z_{\mp}r}k^{\mp}_{1}(w)k_1^{\pm}(z)e^{\pm}_{12}(z)
+\frac{z_{\mp}(s-r)}{w_{\pm}s-z_{\mp}r}k_{1}^{\mp}(w)e^{\mp}_{12}(w)k^{\pm}_{1}(z).
\end{equation*}
Thus from Eq. (\ref{Eq5.18}) we have
\begin{align}\notag
\frac{(w_{\pm}-z_{\mp})rs}{w_{\pm}s-z_{\mp}r}k^{\mp}_{1}(w)k_1^{\pm}(z)e^{\pm}_{12}(z)e^{\mp}_{12}(w)
&=k^{\mp}_{1}(w)e^{\mp}_{12}(w)k^{\pm}_{1}(z)e^{\pm}_{12}(z)\\ \label{Eq5.19}
&-\frac{z_{\mp}(s-r)}{w_{\pm}s-z_{\mp}r}k_{1}^{\mp}(w)e^{\mp}_{12}(w)k^{\pm}_{1}(z)e^{\mp}_{12}(w).
\end{align}
Moreover, using Eq(\ref{Eq4.10}) again, we can obtain
\begin{equation*}
l^{\pm}_{11}(z)l^{\mp}_{12}(w)=\frac{(w_{\pm}-z_{\mp})}{w_{\pm}s-z_{\mp}r}l^{\mp}_{12}(w)l^{\pm}_{11}(z)
+\frac{w_{\pm}(s-r)}{w_{\pm}s-z_{\mp}r}l^{\mp}_{11}(w)l^{\pm}_{12}(z),
\end{equation*}
which is equivalent to
\begin{equation*}
k^{\pm}_1(z)k^{\mp}_1(w)e^{\mp}_{12}(w)=\frac{(w_{\pm}-z_{\mp})}{w_{\pm}s-z_{\mp}r}k^{\mp}_{1}(w)e^{\mp}_{12}(w)k_1^{\pm}(z)
+\frac{w_{\pm}(s-r)}{w_{\pm}s-z_{\mp}r}k_{1}^{\mp}(w)k^{\pm}_{1}(z)e^{\pm}_{12}(z).
\end{equation*}
Then the right hand side of Eq. (\ref{Eq5.19}) equals to
\begin{align*}
&\frac{w_{\pm}s-z_{\mp}r}{w_{\pm}-z_{\mp}}k^{\pm}_1(z)k^{\mp}_1(w)e^{\mp}_{12}(w)e^{\pm}_{12}(z)+
\frac{w_{\pm}(s-r)}{w_{\pm}-z_{\mp}}k^{\mp}_1(w)k^{\pm}_1(z)e^{\pm}_{12}(z)^2\\
&-\frac{z_{\mp}(s-r)}{w_{\pm}-z_{\mp}}k^{\pm}_1(z)k^{\mp}_1(w)e^{\mp}_{12}(w)^2+
\frac{z_{\mp}(s-r)}{w_{\pm}s-z_{\mp}r}\frac{w_{\pm}(s-r)}{w_{\pm}-z_{\mp}}k^{\mp}_1(w)k^{\pm}_1(z)e^{\pm}_{12}(z)e^{\mp}_{12}(w).
\end{align*}
Thus from Eqs. (\ref{Eq5.19}-\ref{Eq5.6}) we prove Eq. (\ref{Eq2_lemma5.8}).
\end{proof}

From lemma \ref{lemma5.8}, we can obtain the following proposition.
\begin{proposition}
\begin{align}\label{Eq1_prop5.9}
(zs-wr)X_1^{+}(z)X_{1}^{+}(w)&=(zr-ws)X_{1}^{+}(w)X_{1}^{+}(z),\\ \label{Eq2_prop5.9}
(ws-zr)X_1^{-}(z)X_{1}^{-}(w)&=(wr-zs)X_{1}^{-}(w)X_{1}^{-}(z),
\end{align}
\end{proposition}
\begin{proof} We prove Eq. (\ref{Eq1_prop5.9}) to illustrate the idea.
From Eq. (\ref{Eq1_lamma5.8}) we have that
\begin{align*}
\frac{zs-wr}{z-w}e_{12}^{+}(z_{+})e_{12}^{+}(w_{+})-\frac{z(s-r)}{z-w}e_{12}^{+}(w_{+})^{2}
&=\frac{ws-zr}{w-z}e_{12}^{+}(w_{+})e_{12}^{+}(z_{+})-\frac{w(s-r)}{w-z}e_{12}^{+}(z_{+})^{2},\\
\frac{zs-wr}{z-w}e_{12}^{-}(z_{-})e_{12}^{-}(w_{-})-\frac{z(s-r)}{z-w}e_{12}^{-}(w_{-})^{2}
&=\frac{ws-zr}{w-z}e_{12}^{-}(w_{-})e_{12}^{-}(z_{-})-\frac{w(s-r)}{w-z}e_{12}^{-}(z_{-})^{2}.
\end{align*}
Using Eq. (\ref{Eq2_lemma5.8}), we get the following identities:
\begin{align*}
\frac{zs-wr}{z-w}e_{12}^{+}(z_{+})e_{12}^{-}(w_{-})-\frac{z(s-r)}{z-w}e_{12}^{-}(w_{-})^{2}
&=\frac{ws-zr}{w-z}e_{12}^{-}(w_{-})e_{12}^{+}(z_{+})-\frac{w(s-r)}{w-z}e_{12}^{+}(z_{+})^{2},\\
\frac{zs-wr}{z-w}e_{12}^{-}(z_{-})e_{12}^{+}(w_{+})-\frac{z(s-r)}{z-w}e_{12}^{+}(w_{+})^{2}
&=\frac{ws-zr}{w-z}e_{12}^{+}(w_{+})e_{12}^{-}(z_{-})-\frac{w(s-r)}{w-z}e_{12}^{-}(z_{-})^{2}.
\end{align*}

Thus we have that
\begin{equation*}
\frac{zs-wr}{z-w}X_1^{+}(z)X_1^{+}(w)=\frac{ws-zr}{w-z}X_1^{+}(w)X_1^{+}(z),
\end{equation*}
which completes the proof of Eq. (\ref{Eq1_prop5.9}).

\end{proof}

Now we are in a position to give the commutation relations between $e_{12}^{\pm}(z)$ and $f_{21}^{\pm}(w)$.
\begin{lemma}\label{lemma5.10}
\begin{align}\label{Eq1_lemma5.10}
[e_{12}^{\pm}(z),f_{21}^{\pm}(w)]&=\frac{z(s-r)}{w-z}(k_{2}^{\pm}(w)k_{1}^{\pm}(w)^{-1}-k_{2}^{\pm}(z)k_{1}^{\pm}(z)^{-1}),\\ \label{Eq2_lemma5.10}
[e_{12}^{\pm}(z),f_{21}^{\mp}(w)]&=\frac{z_{\mp}(s-r)}{w_{\pm}-z_{\mp}}k_{2}^{\mp}(w)k_{1}^{\mp}(w)^{-1}
-\frac{z_{\pm}(s-r)}{w_{\mp}-z_{\pm}}k_{2}^{\pm}(z)k_{1}^{\pm}(z)^{-1}.
\end{align}
\end{lemma}
\begin{proof} As these two commutation relations are proved similarly, we take Eq. (\ref{Eq2_lemma5.10}) to show the derivation.
Note that from Eq. (\ref{Eq4.14}) we have that
\begin{equation}\label{Eq1_proof5.10}
l_{12}^{\pm}(z)l_{21}'^{\mp}(w)+\frac{z_{\mp}(s-r)}{w_{\pm}s-z_{\mp}r}l_{11}^{\pm}(z)l_{11}'^{\mp}(w)=
l_{21}'^{\mp}(w)l_{12}^{\pm}(z)+\frac{z_{\pm}(s-r)}{w_{\mp}s-z_{\pm}r}l_{22}'^{\mp}(w)l_{22}^{\pm}(z)
\end{equation}
Moreover, Eq. (\ref{Eq4.14}) also gives that
\begin{align}\label{Eq2_proof5.10}
l_{12}^{\pm}(z)l_{22}'^{\mp}(w)&=
\frac{w_{\mp}-z_{\pm}}{w_{\mp}s-z_{\pm}r}l_{22}'^{\mp}(w)l_{12}^{\pm}(z)
-\frac{z_{\mp}(s-r)}{w_{\pm}s-z_{\mp}r}l_{11}^{\pm}(z)l_{12}'^{\mp}(w),\\ \label{Eq3_proof5.10}
l_{21}'^{\mp}(w)l_{11}^{\pm}(z)&=
\frac{w_{\pm}-z_{\mp}}{w_{\pm}s-z_{\mp}r}l_{11}^{\pm}(z)l_{21}'^{\mp}(w)
-\frac{z_{\pm}(s-r)}{w_{\mp}s-z_{\pm}r}l_{22}'^{\mp}(w)l_{21}^{\pm}(z).
\end{align}
Substituting Eqs. (\ref{Eq2_proof5.10}-\ref{Eq3_proof5.10}) into Eq. (\ref{Eq1_proof5.10}),
we get the left hand side of Eq. (\ref{Eq1_proof5.10}):
\begin{equation*}
LHS=\frac{w_{\mp}-z_{\pm}}{w_{\mp}s-z_{\pm}r}k_2^{\mp}(w)^{-1}k_1^{\pm}(z)e_{12}^{\pm}(z)f_{21}^{\mp}(w)
+\frac{z_{\mp}(s-r)}{w_{\pm}s-z_{\mp}r}k_1^{\pm}(z)k_1^{\mp}(w)^{-1},
\end{equation*}
as well as the right hand side:
\begin{equation*}
RHS=\frac{w_{\pm}-z_{\mp}}{w_{\pm}s-z_{\mp}r}k_1^{\pm}(z)k_2^{\mp}(w)^{-1}f_{21}^{\mp}(w)e_{12}^{\pm}(z)
+\frac{z_{\pm}(s-r)}{w_{\mp}s-z_{\pm}r}k_2^{\mp}(w)^{-1}k_2^{\pm}(z).
\end{equation*}
Recall that Lemma \ref{lemma5.3} implies
\begin{equation*}
\frac{w_{\mp}-z_{\pm}}{w_{\mp}s-z_{\pm}r}k_{2}^{\mp}(w)^{-1}k_{1}^{\pm}(z)=\frac{w_{\pm}-z_{\mp}}{w_{\pm}s-z_{\mp}r}k_{1}^{\pm}(z)k_{2}^{\mp}(w)^{-1}.
\end{equation*}
Therefore,
\begin{align}\label{Eq4_proof5.10}
\begin{aligned}
&\frac{w_{\pm}-z_{\mp}}{w_{\pm}s-z_{\mp}r}k_1^{\pm}(z)k_2^{\mp}(w)^{-1}[e_{12}^{\pm}(z),f_{21}^{\mp}(w)]\\
&=
-\frac{z_{\mp}(s-r)}{w_{\pm}s-z_{\mp}r}k_1^{\pm}(z)k_1^{\mp}(w)^{-1}
+\frac{z_{\pm}(s-r)}{w_{\mp}s-z_{\pm}r}k_2^{\mp}(w)^{-1}k_2^{\pm}(z).
\end{aligned}
\end{align}
Moreover, using Eq. (\ref{Eq5.7}) we have that
\begin{align*}
k_{1}^{\pm}(z)^{-1}k_{2}^{\mp}(w)^{-1}=
\frac{w_{\pm}-z_{\mp}}{w_{\pm}s-z_{\mp}r}\frac{w_{\mp}s-z_{\pm}r}{w_{\mp}-z_{\pm}}k_{2}^{\mp}(w)^{-1}k_{1}^{\pm}(z)^{-1}.
\end{align*}
Left multiplying $k_{2}^{\mp}(w)k_{1}^{\pm}(z)^{-1}$ on both sides of Eq. (\ref{Eq4_proof5.10}), we get
Eq. (\ref{Eq2_lemma5.10}).
\end{proof}

\begin{proposition}
\begin{align}\notag
[X^{+}_1(z),X^{-}_1(w)]&=(r-s)\delta_{ij}\{\delta(zw^{-1}(r^{-1}s)^{\frac{c}{2}})k_{i+1}^{-}(w_{+})k_{i}^{-}(w_{+})^{-1}-\\
&\quad\qquad\qquad \delta(zw^{-1}(rs^{-1})^{\frac{c}{2}})k_{i+1}^{+}(z_{+})k_{i}^{+}(z_{+})^{-1}\}
\end{align}
\end{proposition}
\begin{proof}
From Lemma \ref{lemma5.10} it follows that
\begin{align}\label{Eq1_proof5.11}
[e_{12}^{+}(z_{+}),f_{21}^{+}(w_{-})]&=
\frac{z_{+}(s-r)}{w_{-}-z_{+}}(k_{2}^{+}(w_{-})k_{1}^{+}(w_{-})^{-1}-k_{2}^{+}(z_{+})k_{1}^{+}(z_{+})^{-1}),\\ \label{Eq2_proof5.11}
[e_{12}^{-}(z_{-}),f_{21}^{-}(w_{+})]&=
\frac{z_{-}(s-r)}{w_{+}-z_{-}}(k_{2}^{-}(w_{+})k_{1}^{-}(w_{+})^{-1}-k_{2}^{-}(z_{-})k_{1}^{-}(z_{-})^{-1}),\\ \label{Eq3_proof5.11}
[e_{12}^{+}(z_{+}),f_{21}^{-}(w_{+})]&=
\frac{z_{-}(s-r)}{w_{+}-z_{-}}k_{2}^{-}(w_{+})k_{1}^{-}(w_{+})^{-1}-\frac{z_{+}(s-r)}{w_{-}-z_{+}}k_{2}^{+}(z_{+})k_{1}^{+}(z_{+})^{-1},\\
\label{Eq4_proof5.11}
[e_{12}^{-}(z_{-}),f_{21}^{+}(w_{-})]&=
\frac{z_{+}(s-r)}{w_{-}-z_{+}}k_{2}^{+}(w_{-})k_{1}^{+}(w_{-})^{-1}-\frac{z_{-}(s-r)}{w_{+}-z_{-}}k_{2}^{-}(z_{-})k_{1}^{-}(z_{-})^{-1},
\end{align}
where the fractions $\frac1{w_{\mp}-z_{\pm}}$ in Eqs. (\ref{Eq1_proof5.11}-\ref{Eq2_proof5.11}) are regarded as power series of $\frac{w}{z}$ or $\frac{z}{w}$,
and the fractions $\frac1{w_{\pm}-z_{\mp}}$ in Eqs. (\ref{Eq3_proof5.11}-\ref{Eq4_proof5.11}) are expanded as power series of $\frac{z}{w}$.
%
%
%
Note that $X^+(z)=e_{12}^+(z_+)-e_{12}^-(z_-)$ and $X^-(w)=f_{21}^+(w_-)-f_{21}^-(w_+)$, so we get that
\begin{equation*}
\begin{aligned}
&\frac1{r-s}[X^{+}_1(z),X^{-}_1(w)]=\\
&\delta_{ij}\{\delta(zw^{-1}(r^{-1}s)^{\frac{c}{2}})k_{i+1}^{-}(w_{+})k_{i}^{-}(w_{+})^{-1}-
\delta(zw^{-1}(rs^{-1})^{\frac{c}{2}})k_{i+1}^{+}(z_{+})k_{i}^{+}(z_{+})^{-1}\}.
\end{aligned}
\end{equation*}
\end{proof}


\subsection{The case $n=3$}
We shall begin with the case of $n=3$, where the Serre relations appear.

When restricting the generating relations (\ref{generating relation1}), (\ref{generating relation2}) to $E_{ij}\otimes E_{kl}$,
$1\leq i,j\leq 2$, we get that
\begin{align}
R_2(\frac{z}{w})J_1^{\pm}(z)J_{2}^{\pm}(w)&=J_{2}^{\pm}(w)J_1^{\pm}(z)R_2(\frac{z}{w}),\\
R_2(\frac{z_{+}}{w_{-}})J_1^{+}(z)J_{2}^{-}(w)&=J_{2}^{-}(w)J_1^{+}(z)R_2(\frac{z_{-}}{w_{+}}),
\end{align}
where $R_2(z)$ denotes the R-matrix for $n=2$ and
\begin{equation*}
J^{\pm}(z)=\begin{pmatrix}
       1 & 0 \\
       f^{\pm}_{21}(z) & 1 \\
     \end{pmatrix}
     \begin{pmatrix}
       k^{\pm}_1(z) & 0  \\
       0 & k_{2}^{\pm}(z) \\
     \end{pmatrix}
     \begin{pmatrix}
       1 & e_{12}^{\pm}(z) \\
       0 & 1 \\
     \end{pmatrix}.
\end{equation*}
Thus our argument for $n=2$ applies to the entries. Similarly, we consider the equivalent generating relations
(\ref{equivalent generating relation1}), (\ref{equivalent generating relation2}), and
restrict them to
$E_{ij}\otimes E_{kl}$, $2\leq i,j\leq 3$ , then we have
\begin{align*}
\widetilde{J}_{1}^{\pm}(z)^{-1}\widetilde{J}_{2}^{\pm}(w)^{-1}R_{2}(\frac{z}{w})&=
R_{2}(\frac{z}{w})\widetilde{J}_{2}^{\pm}(w)^{-1}\widetilde{J}_{1}^{\pm}(z)^{-1},\\
\widetilde{J}_{1}^{+}(z)^{-1}\widetilde{J}_{2}^{-}(w)^{-1}R_{2}(\frac{z_{+}}{w_{-}}&)=
R_{2}(\frac{z_{-}}{w_{+}})\widetilde{J}_{2}^{-}(w)^{-1}\widetilde{J}_{1}^{+}(z)^{-1},
\end{align*}
where
\begin{equation*}
\widetilde{J}^{\pm}(z)=
\begin{pmatrix}
  1 & 0 \\
  f^{\pm}_{32}(z) & 1 \\
\end{pmatrix}
\begin{pmatrix}
  k_2^{\pm}(z) & 0 \\
  0 & k_{3}^{\pm}(z) \\
\end{pmatrix}
\begin{pmatrix}
  1 & e_{23}^{\pm}(z) \\
  0 & 1 \\
\end{pmatrix}.
\end{equation*}
So the entries also satisfy the commutation relations as the case of $n=2$. Now we need the relations between $k_1^{\pm}(z)$, $e_{12}^{\pm}(z)$, or $f_{21}^{\pm}(z)$
and $k_{3}^{\pm}(z)$, $e_{23}^{\pm}(z)$, or $f_{32}^{\pm}(z)$.

\begin{lemma}\label{lemma6.1}

\begin{equation}
k_{1}^{\pm}(z)k_{3}^{\pm}(w)=k_{3}^{\pm}(w)k_{1}^{\pm}(z)
\end{equation}

\begin{equation}\label{Eq6.4}
\frac{w_{\pm}-z_{\mp}}{w_{\pm}s-z_{\mp}r}k_{1}^{\pm}(z)k_{3}^{\mp}(w)^{-1}=
\frac{w_{\mp}-z_{\pm}}{w_{\mp}s-z_{\pm}r}k_{3}^{\mp}(w)^{-1}k_{1}^{\pm}(z)
\end{equation}

\end{lemma}

\begin{proof}
Here we just check Eq. (\ref{Eq6.4}).
Using relations (\ref{Eq4.14}), we have
\begin{equation*}
\frac{w_{\pm}-z_{\mp}}{w_{\pm}s-z_{\mp}r}l_{11}^{\pm}(z)l_{33}'^{\mp}(w)=
\frac{w_{\mp}-z_{\pm}}{w_{\mp}s-z_{\pm}r}l_{33}'^{\mp}(w)l_{11}^{\pm}(z),
\end{equation*}
which is just equivalent to Eq. (\ref{Eq6.4}) by using the Gauss decomposition.
\end{proof}

In the following lemma, we will give the relations between $k_{1}^{\pm}(z)$ and $e_{23}^{\pm}(z)$, $f_{32}^{\pm}(z)$.
\begin{lemma}\label{lemma6.2}
\begin{align}
k_{1}^{\pm}(z)e_{23}^{\pm}(w)&=e_{23}^{\pm}(w)k_{1}^{\pm}(z),\\ \label{Eq6.6}
k_{1}^{\pm}(z)e_{23}^{\mp}(w)&=e_{23}^{\mp}(w)k_{1}^{\pm}(z),\\
k_{1}^{\pm}(z)f_{32}^{\pm}(w)&=f_{32}^{\pm}(w)k_{1}^{\pm}(z),\\
k_{1}^{\pm}(z)f_{32}^{\mp}(w)&=f_{32}^{\mp}(w)k_{1}^{\pm}(z)
\end{align}
\end{lemma}
\begin{proof} We show Eq. (\ref{Eq6.6}), and the other relations can be proved similarly.
We can obtain the following equation from the equivalent generating relations (\ref{Eq4.14})
\begin{equation*}
\frac{w_{\pm}-z_{\mp}}{w_{\pm}s-z_{\mp}r}l_{11}^{\pm}(z)l_{23}'^{\mp}(w)=
\frac{w_{\mp}-z_{\pm}}{w_{\mp}s-z_{\pm}r}l_{23}'^{\mp}(w)l_{11}^{\pm}(z),
\end{equation*}
which is equivalent to the following equation in terms of Gauss generators
\begin{equation*}
\frac{w_{\pm}-z_{\mp}}{w_{\pm}s-z_{\mp}r}k_{1}^{\pm}(z)e_{23}^{\mp}(w)k_{3}^{\mp}(w)^{-1}=
\frac{w_{\mp}-z_{\pm}}{w_{\mp}s-z_{\pm}r}e_{23}^{\mp}(w)k_{3}^{\mp}(w)^{-1}k_{1}^{\pm}(z).
\end{equation*}
Now using Eq. (\ref{Eq6.4}), we have
\begin{equation*}
\frac{w_{\pm}-z_{\mp}}{w_{\pm}s-z_{\mp}r}k_{1}^{\pm}(z)e_{23}^{\mp}(w)k_{3}^{\mp}(w)^{-1}=
\frac{w_{\pm}-z_{\mp}}{w_{\pm}s-z_{\mp}r}e_{23}^{\mp}(w)k_{1}^{\pm}(z)k_{3}^{\mp}(w)^{-1},
\end{equation*}
which is equivalent to Eq. (\ref{Eq6.6}).\end{proof}

Similarly, we give the relations between $k^{\pm}_{3}(z)$ and $e_{12}(z)$, $f_{21}(z)$ in
the following lemma.

\begin{lemma}\label{lemma6.3}
\begin{align}
k_{3}^{\pm}(z)e_{12}^{\pm}(w)&=e_{12}^{\pm}(w)k_{3}^{\pm}(z),\\
k_{3}^{\pm}(z)e_{12}^{\mp}(w)&=e_{12}^{\mp}(w)k_{3}^{\pm}(z),\\
k_{3}^{\pm}(z)f_{21}^{\pm}(w)&=f_{21}^{\pm}(w)k_{3}^{\pm}(z),\\
k_{3}^{\pm}(z)f_{21}^{\mp}(w)&=f_{21}^{\mp}(w)k_{1}^{\pm}(z).
\end{align}
\end{lemma}

Using Lemmas \ref{lemma6.2} and \ref{lemma6.3}, we can easily get the following proposition.
\begin{proposition}
\begin{align}
k_{1}^{\pm}(z)X^{\pm}_{2}(w)&=X^{\pm}_{2}(w)k_{1}^{\pm}(z),\\ \label{Eq2_prop6.4}
k_{1}^{\pm}(z)X^{\mp}_{2}(w)&=X^{\mp}_{2}(w)k_{1}^{\pm}(z),\\
k_{3}^{\pm}(z)X^{\pm}_{1}(w)&=X^{\pm}_{1}(w)k_{3}^{\pm}(z),\\
k_{3}^{\pm}(z)X^{\mp}_{1}(w)&=X^{\mp}_{1}(w)k_{3}^{\pm}(z).
\end{align}
\end{proposition}

In the following lemma, we give the relations between $e_{12}^{\pm}(z)$, $e_{23}^{\pm}(z)$ and
$f_{32}^{\pm}(z)$, $f_{21}^{\pm}(z)$.
\begin{lemma}\label{lemma6.5}
\begin{align}
[e_{12}^{\pm}(z),f_{32}^{\pm}(w)]&=0,\\ \label{Eq6.18}
[e_{12}^{\pm}(z),f_{32}^{\mp}(w)]&=0,\\
[e_{23}^{\pm}(z),f_{21}^{\pm}(w)]&=0,\\
[e_{23}^{\pm}(z),f_{21}^{\mp}(w)]&=0
\end{align}
\end{lemma}
\begin{proof} These identities are similar, so we only prove Eq. (\ref{Eq6.18}). From relation (\ref{Eq4.14}) it follows that
\begin{equation*}
\frac{w_{\pm}-z_{\mp}}{w_{\pm}s-z_{\mp}r}l_{12}^{\pm}(z)l_{32}'^{\mp}(w)=
\frac{w_{\mp}-z_{\pm}}{w_{\mp}s-z_{\pm}r}l_{32}'^{\mp}(w)l_{12}^{\pm}(z),
\end{equation*}
which is equivalent to the following equation in terms of Gauss generators£º
\begin{equation*}
\frac{w_{\pm}-z_{\mp}}{w_{\pm}s-z_{\mp}r}k_{1}^{\pm}(z)e_{12}^{\pm}(z)k_{3}^{\mp}(w)^{-1}f_{32}^{\mp}(w)=
\frac{w_{\mp}-z_{\pm}}{w_{\mp}s-z_{\pm}r}k_{3}^{\mp}(w)^{-1}f_{32}^{\mp}(w)k_{1}^{\pm}(z)e_{12}^{\pm}(z).
\end{equation*}
Now using Lemmas \ref{lemma6.2} and \ref{lemma6.3}, we have that
\begin{equation*}
\frac{w_{\pm}-z_{\mp}}{w_{\pm}s-z_{\mp}r}k_{1}^{\pm}(z)k_{3}^{\mp}(w)^{-1}e_{12}^{\pm}(z)f_{32}^{\mp}(w)=
\frac{w_{\mp}-z_{\pm}}{w_{\mp}s-z_{\pm}r}k_{3}^{\mp}(w)^{-1}k_{1}^{\pm}(z)f_{32}^{\mp}(w)e_{12}^{\pm}(z).
\end{equation*}
Furthermore, the following is obtained by using Eq (\ref{Eq6.4}).
\begin{equation*}
\frac{w_{\pm}-z_{\mp}}{w_{\pm}s-z_{\mp}r}k_{1}^{\pm}(z)k_{3}^{\mp}(w)^{-1}e_{12}^{\pm}(z)f_{32}^{\mp}(w)=
\frac{w_{\pm}-z_{\mp}}{w_{\pm}s-z_{\mp}r}k_{1}^{\pm}(z)k_{3}^{\mp}(w)^{-1}f_{32}^{\mp}(w)e_{12}^{\pm}(z),
\end{equation*}
which is equivalent to Eq. (\ref{Eq6.18}).
\end{proof}

As a consequence of Lemma \ref{lemma6.5}, we have the following result.
\begin{proposition} One has that
\begin{align}
[X^{+}_1(z),X^{-}_{2}(w)]&=0,\\
[X^{+}_2(z),X^{-}_{1}(w)]&=0.
\end{align}
\end{proposition}

\begin{lemma}\label{lemma6.7}
\begin{equation}\label{Eq1_lemma6.7}
\begin{aligned}
e_{12}^{\pm}(z)e_{23}^{\pm}(w)&=&\frac{w-z}{ws-zr}e_{23}^{\pm}(w)e_{12}^{\pm}(z)+
\frac{w(s-r)}{ws-zr}e_{13}^{\pm}(z)\\
&+&\frac{z(s-r)}{ws-zr}e_{12}^{\pm}(w)e_{23}^{\pm}(w)
-\frac{z(s-r)}{ws-zr}e_{13}^{\pm}(w),
\end{aligned}
\end{equation}
\begin{equation}\label{Eq2_lemma6.7}
\begin{aligned}
e_{12}^{\pm}(z)e_{23}^{\mp}(w)&=&\frac{w_{\pm}-z_{\mp}}{w_{\pm}s-z_{\mp}r}e_{23}^{\mp}(w)e_{12}^{\pm}(z)+
\frac{w_{\pm}(s-r)}{w_{\pm}s-z_{\mp}r}e_{13}^{\pm}(z)\\
&+&\frac{z_{\mp}(s-r)}{w_{\pm}s-z_{\mp}r}e_{12}^{\mp}(w)e_{23}^{\mp}(w)-\frac{z_{\mp}(s-r)}{w_{\pm}s-z_{\mp}r}e_{13}^{\mp}(w).
\end{aligned}
\end{equation}
\end{lemma}
\begin{proof} We prove Eq. (\ref{Eq2_lemma6.7}), as Eq.(\ref{Eq1_lemma6.7}) can be shown similarly.
From the generating relations (\ref{Eq4.14}) it follows that
\begin{align}\label{Eq1_proof6.7}
\begin{aligned}
\frac{w_{\mp}-z_{\pm}}{w_{\mp}s-z_{\pm}r}l_{23}'^{\mp}(w)l_{12}^{\pm}(z)&=&
l_{12}^{\pm}(z)l_{23}'^{\mp}(w)
+\frac{w_{\pm}(s-r)}{w_{\pm}s-z_{\mp}r}l_{13}^{\pm}(z)l_{33}'^{\mp}(w)\\
&+&\frac{z_{\mp}(s-r)}{w_{\pm}s-z_{\mp}r}l_{11}^{\pm}(z)l_{13}'^{\mp}(w).
\end{aligned}
\end{align}
Moreover, using the Gauss decomposition and Lemmas \ref{lemma6.1}-\ref{lemma6.2}, we
get that
\begin{equation*}
\begin{aligned}
\frac{w_{\mp}-z_{\pm}}{w_{\mp}s-z_{\pm}r}l_{23}'^{\mp}(w)l_{12}^{\pm}(z)&=&
-\frac{w_{\mp}-z_{\pm}}{w_{\mp}s-z_{\pm}r}e_{23}^{\mp}(w)k_3^{\mp}(w)^{-1}k_{1}^{\pm}(z)e_{12}^{\pm}(z)\\
&=&-\frac{w_{\pm}-z_{\mp}}{w_{\pm}s-z_{\mp}r}k_{1}^{\pm}(z)e_{23}^{\mp}(w)e_{12}^{\pm}(z)k_3^{\mp}(w)^{-1}
\end{aligned}
\end{equation*}
Therefore Eq. (\ref{Eq1_proof6.7}) and the Gauss decomposition imply Eq.(\ref{Eq2_lemma6.7}).
\end{proof}

The following lemma gives the relations between
$f^{\pm}_{21}(z)$ and $f_{32}^{\pm}(w)$.

\begin{lemma}\label{lemma6.8}
\begin{equation}\label{Eq1_lemma6.8}
\begin{aligned}
f_{32}^{\pm}(w)f_{21}^{\pm}(z)&=&\frac{w-z}{ws-zr}f_{21}^{\pm}(z)f_{32}^{\pm}(w)+
\frac{z(s-r)}{ws-zr}f_{31}^{\pm}(z)\\
&+&\frac{w(s-r)}{ws-zr}f_{32}^{\pm}(w)f_{21}^{\pm}(w)
-\frac{w(s-r)}{ws-zr}f_{31}^{\pm}(w),
\end{aligned}
\end{equation}

\begin{equation}\label{Eq2_lemma6.8}
\begin{aligned}
f_{32}^{\mp}(w)f_{21}^{\pm}(z)&=&\frac{w_{\mp}-z_{\pm}}{w_{\mp}s-z_{\pm}r}f_{21}^{\pm}(z)f_{32}^{\mp}(w)+
\frac{z_{\pm}(s-r)}{w_{\mp}s-z_{\pm}r}f_{31}^{\pm}(z)\\
&+&\frac{w_{\mp}(s-r)}{w_{\mp}s-z_{\pm}r}f_{32}^{\mp}(w)f_{21}^{\mp}(w)
-\frac{w_{\mp}(s-r)}{w_{\mp}s-z_{\pm}r}f_{31}^{\mp}(w).
\end{aligned}
\end{equation}

\end{lemma}

From Lemma \ref{lemma6.7} and Lemma \ref{lemma6.8}, we have the following proposition.
\begin{proposition} One has that
\begin{align}\label{Eq1_prop6.9}
X_{1}^{+}(z)X_2^{+}(w)&=\frac{w-z}{ws-zr}X_2^{+}(w)X_{1}^{+}(z),\\ \label{Eq2_prop6.9}
X_{1}^{-}(z)X_2^{-}(w)&=\frac{ws-zr}{w-z}X_2^{-}(w)X_{1}^{-}(z),
\end{align}
\end{proposition}
\begin{proof}
We only prove Eq.(\ref{Eq1_prop6.9}), and Eq. (\ref{Eq2_prop6.9}) can be proved similarly.

From Lemma \ref{lemma6.7}, we have
\begin{align*}
e_{12}^{+}(z_{+})e_{23}^{+}(w_{+})&=\frac{w-z}{ws-zr}e_{23}^{+}(w_{+})e_{12}^{+}(z_{+})+
\frac{w(s-r)}{ws-zr}e_{13}^{+}(z_{+})\\
&+\frac{z(s-r)}{ws-zr}e_{12}^{+}(w_{+})e_{23}^{+}(w_{+})
-\frac{z(s-r)}{ws-zr}e_{13}^{+}(w_{+}),\\
e_{12}^{-}(z_{-})e_{23}^{-}(w_{-})&=\frac{w-z}{ws-zr}e_{23}^{-}(w_{-})e_{12}^{-}(z_{-})+
\frac{w(s-r)}{ws-zr}e_{13}^{-}(z_{})\\
&+\frac{z(s-r)}{ws-zr}e_{12}^{-}(w_{-})e_{23}^{-}(w_{-})
-\frac{z(s-r)}{ws-zr}e_{13}^{-}(w_{-}),\\
e_{12}^{+}(z_{+})e_{23}^{-}(w_{-})&=\frac{w-z}{ws-zr}e_{23}^{-}(w_{-})e_{12}^{+}(z_{+})+
\frac{w(s-r)}{ws-zr}e_{13}^{+}(z_{+})\\
&+\frac{z(s-r)}{ws-zr}e_{12}^{-}(w_{-})e_{23}^{-}(w_{-})
-\frac{z(s-r)}{ws-zr}e_{13}^{-}(w_{-}),\\
e_{12}^{-}(z_{-})e_{23}^{+}(w_{+})&=\frac{w-z}{ws-zr}e_{23}^{+}(w_{+})e_{12}^{-}(z_{-})+
\frac{w(s-r)}{ws-zr}e_{13}^{-}(z_{-})\\
&+\frac{z(s-r)}{ws-zr}e_{12}^{+}(w_{+})e_{23}^{+}(w_{+})
-\frac{z(s-r)}{ws-zr}e_{13}^{+}(w_{+}),
\end{align*}
and these imply Eq.(\ref{Eq1_prop6.9}).
\end{proof}

We now give the Serre relations.
\begin{proposition}
\begin{align}\notag
&\{X_{1}^{-}(z_1)X_{1}^{-}(z_2)X_{2}^{-}(w)-(r+s)X_{1}^{-}(z_1)X_{2}^{-}(w)X_{1}^{-}(z_2)\\
&+rsX_{2}^{-}(w)X_{1}^{-}(z_1)X_{1}^{-}(z_2)
\}+\{z_1\leftrightarrow z_2\}=0,\\ \notag
&\{rsX_{2}^{-}(z_1)X_{2}^{-}(z_2)X_{1}^{-}(w)-(r+s)X_{2}^{-}(z_1)X_{1}^{-}(w)X_{2}^{-}(z_2)\\
&+X_{1}^{-}(w)X_{2}^{-}(z_1)X_{2}^{-}(z_2)
\}+\{z_1\leftrightarrow z_2\}=0,\\ \notag
&\{rsX_{1}^{+}(z_1)X_{1}^{+}(z_2)X_{2}^{+}(w)-(r+s)X_{1}^{+}(z_1)X_{2}^{+}(w)X_{1}^{+}(z_2)\\ \label{Eq3_prop6.10}
&+X_{2}^{+}(w)X_{1}^{+}(z_1)X_{1}^{+}(z_2)
\}+\{z_1\leftrightarrow z_2\}=0.
\end{align}
\begin{equation}
\begin{aligned}
&\{X_{2}^{+}(z_1)X_{2}^{+}(z_2)X_{1}^{+}(w)-(r+s)X_{2}^{+}(z_1)X_{1}^{+}(w)X_{2}^{+}(z_2)\\
&+rsX_{1}^{+}(w)X_{2}^{+}(z_1)X_{2}^{+}(z_2)
\}+\{z_1\leftrightarrow z_2\}=0
\end{aligned}
\end{equation}

\end{proposition}
\begin{proof}
Here we only prove Eq. (\ref{Eq3_prop6.10}), and the other equations can be proved similarly.

Using Eq. (\ref{Eq1_prop6.9}), we can get
\begin{equation}
\begin{aligned}
&rsX_{1}^{+}(z_1)X_{1}^{+}(z_2)X_{2}^{+}(w)-(r+s)X_{1}^{+}(z_1)X_{2}^{+}(w)X_{1}^{+}(z_2)
+X_{2}^{+}(w)X_{1}^{+}(z_1)X_{1}^{+}(z_2)\\
&=(rs+\frac{(ws-z_1r)(ws-z_2r)}{(w-z_1)(w-z_2)}-(r+s)\frac{ws-z_2r}{w-z_2})X_{1}^{+}(z_1)X_{1}^{+}(z_2)X_{2}^{+}(w).
\end{aligned}
\end{equation}
Moreover, from Eq.(\ref{Eq1_prop5.9}) and Eq(\ref{Eq1_prop6.9}), we have
\begin{equation}
\begin{aligned}
&rsX_{1}^{+}(z_2)X_{1}^{+}(z_1)X_{2}^{+}(w)-(r+s)X_{1}^{+}(z_2)X_{2}^{+}(w)X_{1}^{+}(z_1)
+X_{2}^{+}(w)X_{1}^{+}(z_2)X_{1}^{+}(z_1)\\
&=\frac{z_2r-z_1s}{z_2s-z_1r}(rs+\frac{(ws-z_1r)(ws-z_2r)}{(w-z_1)(w-z_2)}-(r+s)\frac{ws-z_1r}{w-z_1})X_{1}^{+}(z_1)X_{1}^{+}(z_2)X_{2}^{+}(w).
\end{aligned}
\end{equation}
We can easily check that
\begin{align*}
&\frac{z_2r-z_1s}{z_2s-z_1r}(rs+\frac{(ws-z_1r)(ws-z_2r)}{(w-z_1)(w-z_2)}-(r+s)\frac{ws-z_1r}{w-z_1})=\\
&-
(rs+\frac{(ws-z_1r)(ws-z_2r)}{(w-z_1)(w-z_2)}-(r+s)\frac{ws-z_2r}{w-z_2}).
\end{align*}
Thus we get Eq. (\ref{Eq3_prop6.10}).
\end{proof}

\subsection{The general $n$ case}

Now we proceed to the case of general $n$.
Just as the case $n=3$, we first restrict the (\ref{generating relation1}) and (\ref{generating relation2}) to $E_{ij}\otimes E_{kl}$
, $1\geq i,j,k,l\leq n-1$, then we get
\begin{equation*}
R_{n-1}(\frac{z}{w})J_1^{\pm}(z)J_{2}^{\pm}(w)=J_{2}^{\pm}(w)J_1^{\pm}(z)R_{n-1}(\frac{z}{w}),
\end{equation*}

\begin{equation*}
R_{n-1}(\frac{z_{+}}{w_{-}})J_1^{+}(z)J_{2}^{-}(w)=J_{2}^{-}(w)J_1^{+}(z)R_{n-1}(\frac{z_{-}}{w_{+}}),
\end{equation*}
\begin{equation*}
J^{\pm}(z)=\begin{pmatrix}
             1 &  & 0 \\
             f_{21}^{\pm}(z) & \ddots &  \\
              &  & \ddots\\
              & f_{n-1,n-2}^{\pm}(z) & 1 \\
           \end{pmatrix}
           \begin{pmatrix}
             k_{1}^{\pm}(z) &  & 0 \\
              & \ddots &  \\
              &  & \ddots \\
             0 &  & k_{n-1}^{\pm}(z) \\
           \end{pmatrix}
           \begin{pmatrix}
             1 & e_{12}^{\pm}(z) &  \\
              & \ddots & \ddots \\
              &  & e_{n-2,n-1}^{\pm}(z)\\
             0 &  & 1 \\
           \end{pmatrix}.
\end{equation*}

Similarly, restricting the generating relations (\ref{equivalent generating relation1}) and (\ref{equivalent generating relation2})
to $E_{ij}\otimes E_{kl}$, $2\leq i,j,k,l\leq n$, then we have
\begin{equation*}
\widetilde{J}_{1}^{\pm}(z)^{-1}\widetilde{J}_{2}^{\pm}(w)^{-1}R_{2}(\frac{z}{w})=
R_{2}(\frac{z}{w})\widetilde{J}_{2}^{\pm}(w)^{-1}\widetilde{J}_{1}^{\pm}(z)^{-1},
\end{equation*}

\begin{equation*}
\widetilde{J}_{1}^{+}(z)^{-1}\widetilde{J}_{2}^{-}(w)^{-1}R_{2}(\frac{z_{+}}{w_{-}})=
R_{2}(\frac{z_{-}}{w_{+}})\widetilde{J}_{2}^{-}(w)^{-1}\widetilde{J}_{1}^{+}(z)^{-1},
\end{equation*}

\begin{equation*}
J^{\pm}(z)=\begin{pmatrix}
             1 &  & 0 \\
             f_{32}^{\pm}(z) & \ddots &  \\
              &  & \ddots\\
              & f_{n,n-1}^{\pm}(z) & 1 \\
           \end{pmatrix}
           \begin{pmatrix}
             k_{2}^{\pm}(z) &  & 0 \\
              & \ddots &  \\
              &  & \ddots \\
             0 &  & k_{n}^{\pm}(z) \\
           \end{pmatrix}
           \begin{pmatrix}
             1 & e_{23}^{\pm}(z) &  \\
              & \ddots & \ddots \\
              &  & e_{n-1,n}^{\pm}(z)\\
             0 &  & 1 \\
           \end{pmatrix}.
\end{equation*}

By induction, we get all the commutator relations we need except those between
$e_{12}^{\pm}(z)$, $k_{1}^{\pm}(z)$, $f_{21}^{\pm}(z)$ and $e_{n-1,n}^{\pm}(z)$, $k_{n}^{\pm}(z)$, $f_{n,n-1}^{\pm}(z)$.
First, using the Gauss decomposition, we write down $L^{\pm}(z)$ and $L^{\pm}(z)^{-1}$:
\begin{equation*}
L^{\pm}(z)=\begin{pmatrix}
             k_1^{\pm}(z) & k_{1}^{\pm}(z)e_{12}^{\pm}(z) & \vdots & \vdots & \ldots \\
             f_{21}^{\pm}(z)k_1^{\pm}(z) & \vdots & \vdots & \vdots & \ldots \\
             \vdots & \vdots & \vdots & \vdots & \ldots \\
           \end{pmatrix}
\end{equation*}
and
\begin{equation*}
L^{\pm}(z)^{-1}=\begin{pmatrix}
                  \ldots & \vdots & \vdots & \vdots \\
                  \ldots & \vdots & \vdots & -e_{n-1,n}^{\pm}(z)k_{n}^{\pm}(z)^{-1} \\
                  \ldots & \vdots & -k_{n}^{\pm}(z)^{-1}f_{n,n-1}^{\pm}(z) & k_{n}^{\pm}(z)^{-1} \\
                \end{pmatrix}.
\end{equation*}

Then using the generating relations (\ref{equivalent generating relation1'}) and (\ref{equivalent generating relation2'}),
we get the following lemma.
\begin{lemma}\label{lemma6.11}
\begin{align}
k_{1}^{\pm}(z)k_{n}^{\pm}(w)&=k_{3}^{\pm}(w)k_{n}^{\pm}(z),\\ \label{Eq2_lemma6.11}
\frac{w_{\pm}-z_{\mp}}{w_{\pm}s-z_{\mp}r}k_{1}^{\pm}(z)k_{n}^{\mp}(w)^{-1}&=
\frac{w_{\mp}-z_{\pm}}{w_{\mp}s-z_{\pm}r}k_{n}^{\mp}(w)^{-1}k_{1}^{\pm}(z),\\
k_{1}^{\pm}(z)e_{n-1,n}^{\pm}(w)&=e_{n-1,n}^{\pm}(w)k_{1}^{\pm}(z),\\
k_{1}^{\pm}(z)e_{n-1,n}^{\mp}(w)&=e_{n-1,n}^{\mp}(w)k_{1}^{\pm}(z),\\
k_{1}^{\pm}(z)f_{n,n-1}^{\pm}(w)&=f_{n,n-1}^{\pm}(w)k_{1}^{\pm}(z),\\
k_{1}^{\pm}(z)f_{n,n-1}^{\mp}(w)&=f_{n,n-1}^{\mp}(w)k_{1}^{\pm}(z),\\
k_{n}^{\pm}(z)e_{12}^{\pm}(w)&=e_{12}^{\pm}(w)k_{n}^{\pm}(z),\\
k_{n}^{\pm}(z)e_{12}^{\mp}(w)&=e_{12}^{\mp}(w)k_{n}^{\pm}(z),\\
k_{n}^{\pm}(z)f_{21}^{\pm}(w)&=f_{21}^{\pm}(w)k_{n}^{\pm}(z),\\
k_{n}^{\pm}(z)f_{21}^{\mp}(w)&=f_{21}^{\mp}(w)k_{n}^{\pm}(z),
\end{align}
\begin{align}
[e_{12}^{\pm}(z),f_{n,n-1}^{\pm}(w)]&=0,\\
[e_{12}^{\pm}(z),f_{n,n-1}^{\mp}(w)]&=0,\\
[e_{n-1,n}^{\pm}(z),f_{21}^{\pm}(w)]&=0,\\
[e_{n-1,n}^{\pm}(z),f_{21}^{\mp}(w)]&=0,\\
f_{21}^{\pm}(z)f_{n,n-1}^{\pm}(w)&=f_{n,n-1}^{\pm}(w)f_{21}^{\pm}(z),\\
f_{21}^{\pm}(z)f_{n,n-1}^{\mp}(w)&=f_{n,n-1}^{\mp}(w)f_{21}^{\pm}(z),\\
e_{12}^{\pm}(z)e_{n-1,n}^{\pm}(w)&=e_{n-1,n}^{\pm}(w)e_{12}^{\pm}(z), \\ \label{Eq18_lemma6.11}
e_{12}^{\pm}(z)e_{n-1,n}^{\mp}(w)&=e_{n-1,n}^{\mp}(w)e_{12}^{\pm}(z)
\end{align}
\end{lemma}
\begin{proof}
We just prove Eq.(\ref{Eq2_lemma6.11}), as the other relations
can be shown similarly.

From Eq.(\ref{Eq4.14}), we have
\begin{equation*}
\frac{w_{\pm}-z_{\mp}}{w_{\pm}s-z_{\mp}r}l_{11}^{\pm}(z)l_{nn}'^{\mp}(w)=
\frac{w_{\mp}-z_{\pm}}{w_{\mp}s-z_{\pm}r}l_{nn}'^{\mp}(w)l_{11}^{\pm}(z),
\end{equation*}
which is equivalent to Eq.(\ref{Eq2_lemma6.11}) from the Gauss decomposition. \end{proof}

By induction and using Lemma \ref{lemma6.11}, we have proved theorem \ref{Theorem5.2} for the general $n$ case.
\medskip
\section{Drinfeld realization of $U_{r,s}(\widehat{\mathfrak{sl}}_n)$}
In this section, we will give the Drinfeld realization for $U_{r,s}(\widehat{\mathfrak{sl}}_n)$.
Analogue to the one-parameter case, we define $U_{r,s}(\widehat{\mathfrak{sl}}_n)$ as the subalgebra
of $U_{r,s}(\widehat{\mathfrak{gl}}_n)$ generated by
\begin{align*}
x_i^{\pm}(z)&=(r-s)^{-1}X_{i}^{\pm}(z(rs^{-1})^{\frac{i}{2}}),\\ \varphi_i(z)&=k^{+}_{i+1}(z(rs^{-1})^{\frac{i}{2}})k^{+}_{i}(z(rs^{-1})^{\frac{i}{2}})^{-1},\\
\psi_i(z)&=k^{-}_{i+1}(z(rs^{-1})^{\frac{i}{2}})k^{-}_{i}(z(rs^{-1})^{\frac{i}{2}})^{-1}.
\end{align*}

Let
\begin{equation*}
g_{ij}(z)=\sum_{n\in \mathbb{Z}_{+}}c_{ijn}z^n
\end{equation*}
be the formal power series in $z$ such that the coefficients $c_{ijn}$ are determined from the Taylor expansion in the variable $z$ at
$0\in \mathbb{C}$ of the function
\begin{equation*}
f_{ij}(z)=\frac{(rs^{-1})^{\frac{a_{ij}}{2}}z-1}{z-(rs^{-1})^{\frac{a_{ij}}{2}}}.
\end{equation*}

From theorem \ref{Theorem5.2}, we can get the following proposition.
\begin{proposition}\label{prop7.1}
In $U_{r,s}(\widehat{\mathfrak{sl}}_n)$, the generators $x_i^{\pm}(z)$, $\varphi_i(z)$ and $\psi_i(z)$ satisfy the following
relations:
\begin{align}
[\varphi_i(z),\varphi_j(w)]&=0,~~~~~~~[\psi_i(z),\psi_j(w)]=0, \\
\varphi_i(z)\psi_j(w)&=\frac{g_{ij}(\frac{z_{-}}{w_{+}})
}{g_{ij}(\frac{z_{+}}{w_{-}})}\psi_j(w)\varphi_i(z),\\
\varphi_i(z)x_j^{\pm}(w)&=(rs)^{\pm\frac{i-j}{2}}g_{ij}(\frac{z_{\mp}}{w})^{\pm 1}x_j^{\pm}(w)\varphi_i(z),~~~~~~~~~~~~~~~~|i-j|\leq 1\\
\varphi_i(z)x_j^{\pm}(w)&=x_j^{\pm}(w)\varphi_i(z),~~~~~~~~~~~~~~~~|i-j|> 1,
\end{align}
\begin{align}
\psi_i(z)x_j^{\pm}(w)&=(rs)^{\mp\frac{i-j}{2}}g_{ij}(\frac{w}{z_{\pm}})^{\mp1}x_j^{\pm}(w)\psi_i(z),~~~~~~~~~~~~~~~~|i-j|\leq 1,\\
\psi_i(z)x_j^{\pm}(w)&=x_j^{\pm}(w)\psi_i(z),~~~~~~~~~~~~~~~~|i-j|> 1,\\
(z-w(rs^{-1})^{\pm\frac{a_{ij}}{2}})x^{\pm}_i(z)x^{\pm}_j(w)&=(rs)^{\pm\frac{i-j}{2}}(z(rs^{-1})^{\pm\frac{a_{ij}}{2}}-w)x^{\pm}_j(w)x^{\pm}_i(z),~~~~~~~~~~~~~~~~|i-j|\leq 1,\\
x^{\pm}_i(z)x^{\pm}_j(w)&=x^{\pm}_j(w)x^{\pm}_i(z),~~~~~~~~~~~~~~~~|i-j|> 1,
\end{align}
\begin{align}
&[x^{+}_i(z),x^{-}_j(w)]=(r-s)^{-1}\delta_{ij}\{\delta(\frac{z_{-}}{w_{+}})\psi_i(w_{+})-
\delta(\frac{z_{+}}{w_{-}})\varphi_i(z_+)\},\\ \notag
&\{x_{i}^{\pm}(z_1)x_{i}^{\pm}(z_2)x_{i+1}^{\pm}(w)-(r^{\mp 1}+s^{\mp 1})x_{i}^{\pm}(z_1)x_{i+1}^{\pm}(w)x_{i}^{\pm}(z_2)\\
&+(rs)^{\mp 1}x_{i+1}^{\pm}(w)x_{i}^{\pm}(z_1)x_{i}^{\pm}(z_2)
\}+\{z_1\leftrightarrow z_2\}=0,\\ \notag
&\{x_{i+1}^{\pm}(z_1)x_{i+1}^{\pm}(z_2)x_{i}^{\pm}(w)-(r^{\pm 1}+s^{\pm 1})x_{i+1}^{\pm}(z_1)x_{i}^{\pm}(w)x_{i+1}^{\pm}(z_2)\\
&+(rs)^{\pm 1}x_{i}^{\pm}(w)x_{i+1}^{\pm}(z_1)x_{i+1}^{\pm}(z_2)
\}+\{z_1\leftrightarrow z_2\}=0
\end{align}
\end{proposition}


\begin{remark}
Proposition \ref{prop7.1} can be viewed as the two-parameter analogue of the Drinfeld realization of $U_q(\mathfrak{\widehat{sl}}_n)$.
When $r=q=s^{-1}$, it degenerates into the Drinfeld realization of $U_q(\mathfrak{\widehat{sl}}_n)$.
\end{remark}








\centerline{\bf Acknowledgments}
The authors are indebted to the support of
Simons Foundation grant 198129, NSFC grant nos. 11271138 and 11531004.

\bibliographystyle{amsalpha}

\end{document}